\documentclass[a4paper,11pt,reqno]{amsart}
\usepackage{CJKutf8}

\usepackage{tikz}
\usetikzlibrary{decorations.markings}

\usepackage{amsmath,amsthm,amssymb,amsfonts,amsbsy}
\usepackage{enumitem,color,graphicx,cite}
\usepackage[all,pdf]{xy}

\usepackage[dvipsnames]{xcolor}

\usepackage{geometry}
\usepackage[backref=page,hyperindex=true,CJKbookmarks=true,
colorlinks,linkcolor=blue,anchorcolor=red,citecolor=cyan]{hyperref}

\theoremstyle{plain}
\newtheorem{theorem}{Theorem}[section]
\newtheorem{proposition}[theorem]{Proposition}
\newtheorem{lemma}[theorem]{Lemma}
\newtheorem{corollary}[theorem]{Corollary}

\theoremstyle{remark}
\newtheorem{remark}[theorem]{Remark}

\theoremstyle{definition}
\newtheorem{definition}[theorem]{Definition}

\newtheorem{claim}[theorem]{Claim}

\theoremstyle{plain}
\newtheorem{maintheorem}{Theorem}

\numberwithin{equation}{section}

\renewcommand*{\backref}[1]{}
\renewcommand*{\backrefalt}[4]{%
	\ifcase #1 (Not cited.)%
	\or        (Cited on page~#2.)%
	\else      (Cited on pages~#2.)%
	\fi}

\address[W. Li]{College of Mathematics, 
                Sichuan University, 
                Chengdu, Sichuan 610065, P.R. China}
\email{lwc@scu.edu.cn}

\address[Y. Shi]{College of Mathematics, 
                Sichuan University, 
                Chengdu, Sichuan 610065, P.R. China}
\email{shiyi@scu.edu.cn}

\address[M. Xia]{School of Mathematical Sciences,
                Dalian University of Technology,
                Dalian, Liaoning 116024, P.R. China}
\email{xiamingyang@dlut.edu.cn}

\title[Robust Transitivity]{Robust Transitivity of Partially Hyperbolic Diffeomorphisms with Interval Central Leaves}

\author[W. Li, Y. Shi and M. Xia]{Wenchao LI, Yi SHI and Mingyang XIA}

\date{\today}


\begin{document}
\begin{CJK}{UTF8}{gbsn}

\begin{abstract}
	\begin{sloppypar}
	For a boundary-preserving partially hyperbolic diffeomorphism with interval central leaves, 
    we completely characterize the $C^k$-robust transitivity $(k\geq 2)$ by boundary interconnection. 
    As an application, if the boundary SRB measures admit negative central Lyapunov exponents, 
    then boundary interconnection also completely characterizes the phenomenon of robustly intermingled basins for boundary SRB measures.
    \end{sloppypar}
\end{abstract}

\maketitle

\section{Introduction}\label{section: introduction}
As one of the most studied topological properties of dynamical systems, 
transitivity has been a cornerstone of chaotic behavior. 
The transitivity of a dynamical system implies the indecomposability of the system.
Robust transitivity requires the system to remain transitive under small perturbations. 
Since we can never measure the system with infinite precision in the physical world, 
robust transitivity produces a stable and observable phenomenon of chaos.

The simplest and most well known systems with robust transitivity are transitive Anosov diffeomorphisms. 
They are robustly transitive due to structural stability \cite{Anosov1967}. 
Although the question remains widely open as to which manifold can support an Anosov diffeomorphism 
and whether Anosov diffeomorphisms are all transitive \cite{Smale1967}, 
all of the discovered Anosov diffeomorphisms are transitive and supported on infra-nilmanifolds \cite{Ma74}, 
which are finitely covered by nilmanifolds.

It then becomes a question whether robust transitivity is completely characterized by hyperbolicity, 
in a uniform sense or a weaker sense. 
In 1971, Shub \cite{Sh71} constructed examples of robustly transitive diffeomorphisms with partial hyperbolicity on $\mathbb{T}^4$. 
In 1978, Ma\~n\'e \cite{Mn78} constructed robustly transitive examples on $\mathbb{T}^3$ with partial hyperbolicity. 
In 1996, Bonatti and D\'iaz \cite{BD96} introduced the powerful mechanism of blender to construct robustly transitive systems from perturbations. 
In 2000, Bonatti and Viana \cite{BV00} produced further examples that are not partially hyperbolic but admit a dominated splitting. 
These examples indicate that generally robust transitivity does not imply uniform hyperbolicity, or even partial hyperbolicity.

However, robust transitivity indeed implies some hyperbolic behavior of the dynamical system. 
In 1982, Ma\~n\'e \cite{Mane1982} proved that $C^1$-robust transitivity implies uniform hyperbolicity in the case of surface diffeomorphisms. 
In 1999, D\'iaz, Pujals and Ures \cite{DPU99} showed that $C^1$-robust transitivity implies partial hyperbolicity in dimension three. 
Finally, in 2003, Bonatti, D\'iaz, and Pujals \cite{BDP03} proved that $C^1$-robustly transitive diffeomorphisms on closed manifolds must be volume partially hyperbolic.

In this paper, we give a complete characterization of $C^k$-robust transitivity ($k\geq2$) for the simplest partially hyperbolic diffeomorphisms, 
i.e. the interval extensions of Anosov diffeomorphisms on nilmanifolds. 

Let $N$ be a nilmanifold and $M = N\times [0, 1]$ be the thickened nilmanifold. 
Note that $\partial M$ is a disjoint union of $M_0 := N\times\{0\}$ and $M_1 := N\times\{1\}$. 
Let $\mathrm{Diff}_\partial^k(M)$ be the collection of $C^k$ diffeomorphisms on $M$ preserving $M_0$ and $M_1$ respectively, 
equipped with the usual $C^k$ topology.

\begin{definition}\label{definition: RT}
    We say that $F\in\mathrm{Diff}_\partial^k(M)$ is \textit{$C^k$-robustly transitive}, 
    if there exists an open neighborhood $\mathcal{U}\subseteq \mathrm{Diff}_\partial^k(M)$ of $F$ 
    such that every $G\in\mathcal{U}$ is transitive.
\end{definition}

The study of boundary-preserving systems has its own interests. 
Pac\'\i fico \cite{Pac84} studied the structural stability of vector fields on $3$-manifolds with boundary. 
Ilyashenko \cite{Il11} also studied the quasi-open set in the space of boundary-preserving diffeomorphisms of $\mathbb{T}^2\times[0,1]$, 
formed by the maps exhibiting the different types of thick attractors which is almost topologically mixing.

\begin{definition}\label{definition: PHI}
    We say that $F\in\mathrm{Diff}_\partial^k(M)$ is a \textit{partially hyperbolic diffeomorphism with interval central leaves}, 
    if $F$ satisfies the following properties:
    \begin{itemize}
        \item $F$ is \textit{partially hyperbolic}, 
        i.e. there exists a continuous $DF$-invariant splitting $TM = E^s\oplus E^c \oplus E^u$ such that for any $x\in M$,
        \begin{align*}
	       &\|DF|_{E^s(x)}\| < 1 < m(DF|_{E^u(x)}), \\
	       &\|DF|_{E^s(x)}\| < m(DF|_{E^c(x)}) \leq \|DF|_{E^c(x)}\|< m(DF|_{E^u(x)}).
        \end{align*}
        As a result, there are two $F$-invariant foliations, 
        the stable foliation $\mathcal{F}^s$ and the unstable foliation $\mathcal{F}^u$, tangent to $E^s$ and $E^u$ respectively.
        
        \item $F$ is \textit{dynamically coherent}, 
        i.e. there are two $F$-invariant foliations, 
        the center-stable foliation $\mathcal{F}^{cs}$ and the center-unstable foliation $\mathcal{F}^{cu}$, 
        tangent to $E^{cs} := E^c\oplus E^s$ and $E^{cu} := E^c\oplus E^u$ respectively. 
        As a result, their intersection, the central foliation $\mathcal{F}^c$, is an $F$-invariant foliation tangent to $E^c$.
        
        \item Each central leaf $\mathcal{F}^c(\cdot)$ is an interval that intersects transversely with $M_0$ and $M_1$ exactly at its endpoints.
    \end{itemize}
    The collection of partially hyperbolic diffeomorphisms with interval central leaves in $\mathrm{Diff}_\partial^k(M)$ will be denoted by $\mathrm{PHI}^k(M)$.
\end{definition}

\begin{remark}\label{remark: PHI}
    Since $E^c$ is uniformly transverse to the boundary $M_0$ and $M_1$, 
    $\mathrm{PHI}^k(M)$ is open in $\mathrm{Diff}_\partial^k(M)$ \cite{HPS77}. 
    For $F\in\mathrm{PHI}^k(M)$ and $x\in M_i\ (i = 0,1)$, 
    since $M_i$ is $F$-invariant and transverse to $E^c(x)$, we have 
    $$
    T_xM = T_xM_i\oplus E^c(x),
    \text{ where } 
    T_xM_i=E^s(x)\oplus E^u(x).
    $$ 
    In particular, $F_i := F|_{M_i}:M_i\to M_i$ is an Anosov diffeomorphism. 
    Let $\pi_i: M \to M_i$ be the natural projection along central leaves, defined by $\pi_i(x) := \mathcal{F}^c(x)\pitchfork M_i$. 
    Then we have $\pi_i\circ F = F_i\circ\pi_i$. 
    In particular, we have $\pi_0|_{M_1}\circ F_1 = F_0\circ \pi_0|_{M_1}$, where $\pi_0|_{M_1}: M_1 \to M_0$ is a homeomorphism. 
    Therefore, $F_0$ and $F_1$ are topologically conjugated.
\end{remark}

\begin{definition}\label{definition: BI}
    We say that $F\in\mathrm{PHI}^k(M)$ is \textit{boundary interconnected}, 
    if there exist two pairs of periodic points $p_i,q_i\in M_i\ (i = 0,1)$ such that
    \begin{align*}
        \lambda^c(p_0) < 0 < \lambda^c(p_1) &\text{ and } W^s(p_0)\cap W^u(p_1)\neq\varnothing;\\
        \lambda^c(q_1) < 0 < \lambda^c(q_0) &\text{ and } W^s(q_1)\cap W^u(q_0)\neq\varnothing.
    \end{align*}
    Here $\lambda^c(\cdot)$ is the central Lyapunov exponent 
    at the given periodic point which is the logarithm of the absolute value of corresponding eigenvalue of the derivative 
    restricted to $E^c$, divided by the period;
    $W^s(\cdot)$ and $\ W^u(\cdot)$ are the stable and unstable sets, respectively.
\end{definition}

\begin{remark}\label{remark: BI}
    In fact, $W^s(p_0)$ is an open submanifold of $\mathcal{F}^{cs}(p_0)$, 
    and $W^u(p_1)$ is an open submanifold of $\mathcal{F}^{cu}(p_1)$, 
    given by the hyperbolicity of the periodic points (see Lemma \ref{lemma: W^s and W^u}). 
    Hence $\dim W^s(p_0) + \dim W^u(p_1) = \dim M + 1$, 
    and their intersection $W^s(p_0)\pitchfork W^u(p_1)$ consists of central intervals. 
    Similarly, $W^s(q_1)\pitchfork W^u(q_0)$ also consists of central intervals.
\end{remark}

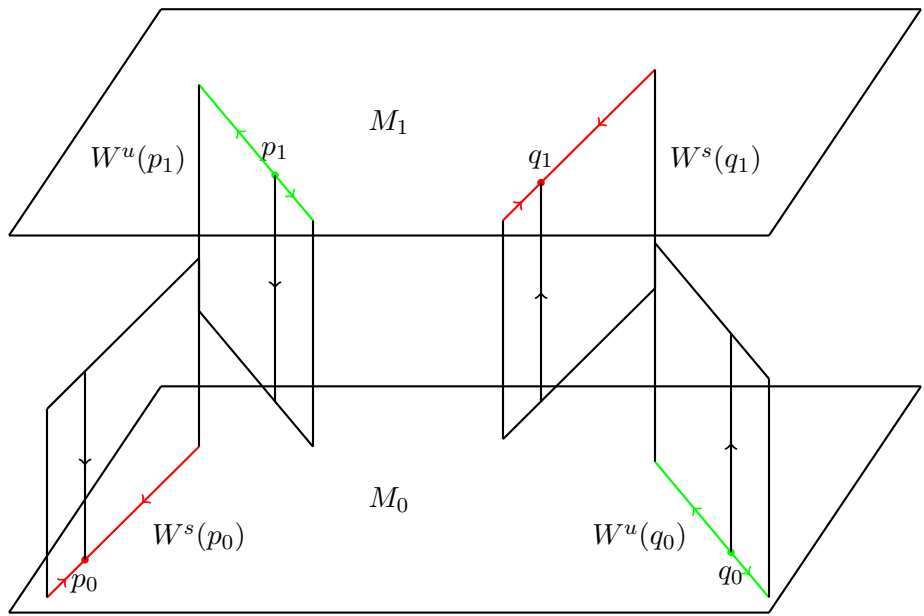
\begin{figure}[htbp]
	\centering
    \begin{tikzpicture}[mid arrow/.style = {decoration = {markings, mark = at position 0.5 with {\arrow{>}}}, postaction = {decorate}}]
        \draw [thick] (0,0)--(10,0);
        \draw [thick] (0,0)--(2,3);
        \draw [thick] (10,0)--(12,3);
        \draw [thick] (2,3)--(12,3);
        \node [label = center: $M_0$] at (5,1.5) {};
        
        \draw [thick] (0,5)--(10,5);
        \draw [thick] (0,5)--(2,8);
        \draw [thick] (10,5)--(12,8);
        \draw [thick] (2,8)--(12,8);
        \node [label = center: $M_1$] at (5,6.5) {};
        
        \node [circle, fill = red, inner sep = 1pt, label = below: $p_0$] at (1,0.7) {};
        \draw [thick, mid arrow, color = red](0.5,0.2)--(1,0.7);
        \draw [thick, mid arrow, color = red](2.5,2.2)--(1,0.7);
        \draw [thick] (0.5,0.2)--(0.5,2.7);
        \draw [thick] (2.5,2.2)--(2.5,4.7);
        \draw [thick] (0.5,2.7)--(2.5,4.7);
        \draw [thick, mid arrow](1,3.2)--(1,0.7);
        \node [label = center: $W^s(p_0)$] at (2.5,1) {};
        
        \node [circle, fill = green, inner sep = 1pt, label = above: $p_1$] at (3.5,5.8) {};
        \draw [thick, mid arrow, color = green](3.5,5.8)--(4,5.2);
        \draw [thick, mid arrow, color = green](3.5,5.8)--(2.5,7);
        \draw [thick] (2.5,7)--(2.5,4);
        \draw [thick] (4,5.2)--(4,2.2);
        \draw [thick] (2.5,4)--(4,2.2);
        \draw [thick, mid arrow] (3.5,5.8)--(3.5,2.8);
        \node [label = center: $W^u(p_1)$] at (1.7,6) {};
        
        \node [circle, fill = green, inner sep = 1pt, label = below: $q_0$] at (9.5,0.8) {};
        \draw [thick, mid arrow, color = green](9.5,0.8)--(10,0.2);
        \draw [thick, mid arrow, color = green](9.5,0.8)--(8.5,2);
        \draw [thick] (8.5,4.9)--(8.5,2);
        \draw [thick] (10,3.1)--(10,0.2);
        \draw [thick] (8.5,4.9)--(10,3.1);
        \draw [thick, mid arrow] (9.5,0.8)--(9.5,3.7);
        \node [label = center: $W^u(q_0)$] at (8.3,1) {};
        
        \node [circle, fill = red, inner sep = 1pt, label = above: $q_1$] at (7,5.7) {};
        \draw [thick, mid arrow, color = red](6.5,5.2)--(7,5.7);
        \draw [thick, mid arrow, color = red](8.5,7.2)--(7,5.7);
        \draw [thick] (6.5,2.3)--(6.5,5.2);
        \draw [thick] (8.5,4.3)--(8.5,7.2);
        \draw [thick] (6.5,2.3)--(8.5,4.3);
        \draw [thick, mid arrow](7,2.8)--(7,5.7);
        \node [label = center: $W^s(q_1)$] at (9.3,6) {};
    \end{tikzpicture}
	\caption{Boundary Interconnection}
	\label{figure: BI}
\end{figure}

\subsection{Robust Transitivity}
In this paper, we obtain the following characterization of robust transitivity for partially hyperbolic diffeomorphisms with interval central leaves.

\begin{maintheorem}\label{theorem: main theorem A}
	For $F\in\mathrm{PHI}^k(M)\ (k \geq 2)$, the following statements are equivalent:
	\begin{enumerate}
	    \item $F$ is $C^k$-robustly transitive;
        \item $F$ is boundary interconnected.
	\end{enumerate}
\end{maintheorem}

\begin{remark}\label{remark: main theorem A}
    Note that the boundary interconnection property is $C^1$-open in $\mathrm{PHI}^1(M)$. 
    We actually prove that for $F\in\mathrm{PHI}^1(M)$, if $F$ is boundary interconnected, 
    then there exists an open neighborhood $\mathcal{U}\subseteq \mathrm{PHI}^1(M)$ of $F$, 
    such that every $G\in\mathcal{U}\cap \mathrm{PHI}^2(M)$ is topologically transitive, see Corollary \ref{corollary: C1 and C2}. 
    This kind of robust transitivity is exactly similar to the concept of stable ergodicity: 
    It relies on a $C^1$-robust structure provided by a $C^2$ partially hyperbolic diffeomorphism, 
    which exactly corresponds to the role of the absolute continuity in the stably ergodic issue with Hopf's argument.
\end{remark}

\begin{remark}\label{remark: base structure}
    The geometric and algebraic structure of $M$ is crucial for Theorem \ref{theorem: main theorem A}. 
    In fact, the conclusion does not hold when $F_i: M_i \to M_i$ is a general uniformly hyperbolic system, 
    rather than an Anosov diffeomorphism on a nilmanifold. 
    For example, let $\sigma: \Sigma \to \Sigma$ be the classical horseshoe, 
    where $\Sigma = \{0,1\}^\mathbb{Z}$ and $\sigma(x)_i = x_{i + 1}$. 
    Assume that $\Phi: [0,1] \to [0,1]$ is a smooth diffeomorphism with exactly two fixed points, a source at $1$ and a sink at $0$. 
    Assume further that $\Phi'$ is dominated by the hyperbolicity of $\sigma$, 
    and consider the partially hyperbolic system $F: \Sigma \times[0,1] \to \Sigma\times[0,1]$, $F(x,t) = (\sigma(x), \Phi_x(t))$, 
    where $\Phi_x = \Phi$ when $x_0 = 0$ and $\Phi_x = \Phi^{-1}$ when $x_0 = 1$. 
    Then $F$ is boundary interconnected but not topologically transitive. 
    In fact, choose two fixed points $p = (0),\ q = (1)$ of $\sigma$, 
    then $p_0 :=(p,0)$, $p_1 := (p,1)$, $q_0 := (q,0)$, and $q_1 := (q,1)$ satisfy the property that we need for boundary interconnection. 
    However, if we take open intervals $U,\ V\subseteq (0,\ 1)$ such that $\Phi^n(U)\cap V = \varnothing$ for any $n\in\mathbb{Z}$, 
    then $\Sigma\times \bigcup_{n\in\mathbb{Z}}\Phi^n(U)$ and $\Sigma\times \bigcup_{n\in\mathbb{Z}}\Phi^n(V)$ are disjoint $F$-invariant open subsets, 
    and hence $F$ is not topologically transitive.
\end{remark}

\subsection{Robustly Intermingled Basins}
The phenomenon of intermingled basins was originally investigated by Kan \cite{Kan94}, 
and summarized by Bonatti, D\'iaz and Viana \cite{BDV05}. 
In such a system, there are two SRB (physical) measures, 
both of whose basins admit positive volume in every non-empty open subset. 
Some of these examples are topologically transitive or even topologically mixing \cite{GS19,Xia23}, 
indicating the significant difference between the topological and measure-theoretical notions of indecomposability. 
Moreover, the Kan's intermingled basins are robust \cite{KS11}.

This type of construction continues to be a source of interesting researches and has been especially studied in the aspect of intermingled basins \cite{MW05,DVY16,BP18,UV18}. 
There is also a family of topologically transitive diffeomorphisms constructed by inserting a blender and embedding into a boundaryless manifold \cite{CGS18}. 
More recently, there have been some advances focusing on the number of measures of maximal entropy for systems related to this type of construction \cite{NBRV21,RT22}.

Specifically, we describe the phenomenon of intermingled basins 
in the setting of partially hyperbolic diffeomorphisms with interval central leaves.

Given $F\in\mathrm{PHI}^k(M)\ (k \geq 2)$, for $i = 0,1$, 
there exists a unique SRB measure $\mu_i$ on $M_i$ for $F_i$, 
since $F_i$ is an Anosov diffeomorphism on a nilmanifold $M_i$ and therefore topologically transitive. 
We call them \textit{the boundary SRB measures} of $F$.

Denote by $\mathrm{PHI}_-^k(M)\ (k \geq 2)$ the collection of diffeomorphisms $F\in\mathrm{PHI}^k(M)$ 
that are mostly contracting along the center direction on each boundary, 
i.e. both of the boundary SRB measures $\mu_i\ (i = 0,1)$ admit negative central Lyapunov exponents, 
in the sense that
$$
\int_{M_i}\ln\left\|DF|_{E^c}\right\|d\mu_i < 0.
$$
Note that $\mathrm{PHI}^k_-(M)$ is open in $\mathrm{PHI}^k(M)$, 
and for $F\in\mathrm{PHI}^k_-(M)$, 
$\mu_i\ (i = 0,1)$ is not only an SRB measure for $F_i$, 
but also a physical measure for $F$.

\begin{definition}\label{definition: IB and RIB}
    We say that $F\in\mathrm{PHI}_-^k(M)$ admits \textit{intermingled basins}, 
    if the basins of its boundary SRB (physical) measures, $B(\mu_0)$ and $B(\mu_1)$, 
    both admit positive volume inside any non-empty open subset of $M$, and their union has full volume. 
    Further, we say that $F$ admits \textit{$C^k$-robustly intermingled basins}, 
    if there exists an open neighborhood $\mathcal{U}\subseteq\mathrm{PHI}_-^k(M)$ of $F$ 
    such that every $G\in\mathcal{U}$ admits intermingled basins.
\end{definition}

An interesting result of Ures and V\'asquez \cite{UV18} showed that 
the phenomenon of robustly intermingled basins does not appear on $\mathbb{T}^3$, 
indicating the particularity of manifolds with boundary.

In this paper, we also obtain the characterization of robustly intermingled basins 
for partially hyperbolic diffeomorphisms with interval central leaves 
that are mostly contracting along the center direction on each boundary.

\begin{maintheorem}\label{theorem: main theorem B}
    For $F\in\mathrm{PHI}_-^k(M)\ (k \geq 2)$, the following statements are equivalent:
    \begin{enumerate}
        \item $F$ admits $C^k$-robustly intermingled basins;
        \item $F$ is $C^k$-robustly transitive;
        \item $F$ is boundary interconnected.
    \end{enumerate}
\end{maintheorem}

As an application of Theorem \ref{theorem: main theorem A} and Theorem \ref{theorem: main theorem B}, 
consider the skew-product system
$$
A_\varphi(x,t) := (Ax,\, \varphi(x, t)): \mathbb{T}^d\times[0, 1] \to \mathbb{T}^d\times[0, 1],
$$
where $A\in\mathrm{Aut}(\mathbb{T}^d)$ is a hyperbolic toral automorphism 
with hyperbolic splitting $L^s\oplus L^u$, 
and $\varphi(x, t): \mathbb{T}^d\times[0, 1] \to [0, 1]$ is a $C^k\ (k \geq 2)$ map 
such that $\varphi(x, \cdot) \in \mathrm{Diff}_\partial^k([0, 1])$ for each $x\in \mathbb{T}^d$. 
Assume that
$$
\|A|_{L^s}\| < \left|\frac{\partial \varphi}{\partial t}\right| < m(A|_{L^u}).
$$
Then $A_\varphi\in\mathrm{PHI}^k(\mathbb{T}^d\times[0, 1])$. 
For convenience, we denote the iterations of $A_\varphi$ by
$$
A_\varphi^n(x,\, t) = (A^nx,\, \varphi_x^n(t)).
$$

Following the definition proposed by Bonatti, D\'iaz and Viana in \cite[Chapter 11.1.1]{BDV05},
we say that $A_\varphi$ is a \textit{Kan-type skew-product system}, 
if there exist two periodic points $p$ and $q$ of $A\in\mathrm{Aut}(\mathbb{T}^d)$ with periods $\pi(p)$ and $\pi(q)$, 
such that $\varphi_p^{\pi(p)}$ has exactly two fixed points, a source at $(p,1)$ and a sink at $(p,0)$; 
$\varphi_q^{\pi(q)}$ also has exactly two fixed points, a source at $(q,0)$ and a sink at $(q,1)$. 
Further, we say that $A_\varphi$ is mostly contracting, if
$$
\int_{\mathbb{T}^d}{\ln}\frac{\partial \varphi}{\partial t}(x,0)\,dx < 0 
\quad\text{ and }\quad
\int_{\mathbb{T}^d}{\ln}\frac{\partial \varphi}{\partial t}(x,1)\,dx < 0.
$$

By Theorem \ref{theorem: main theorem A}, Kan-type skew-product systems are $C^k$-robustly transitive. 
By Theorem \ref{theorem: main theorem B}, mostly contracting Kan-type skew product systems admit $C^k$-robustly intermingled basins, 
for which the boundary SRB measures coincide with the Lebesgue measure restricted to each boundary component, respectively.

\vspace{0.2cm}
\noindent
{\bf Organization.} This paper is organized as follows.
	In Section \ref{section: preliminaries}, we introduce some basic definitions and results. 
    In Section \ref{section: BI to TT}, we establish topological transitivity from boundary interconnection (Theorem \ref{theorem: BI to TT}), 
        which exactly shows $C^2$-robust transitivity in Theorem \ref{theorem: main theorem A}. 
    Then we complete the proof of Theorem \ref{theorem: main theorem A} in Section \ref{section: RTT to BI}, 
        showing that $C^\infty$-robust transitivity implies boundary interconnection. 
    In Section \ref{section: IB}, we prove Theorem \ref{theorem: main theorem B} 
        by establishing the equivalence between robustly intermingled basins and boundary interconnection.

\section{Preliminaries}\label{section: preliminaries}

\subsection{Partially hyperbolic diffeomorphisms with interval central leaves}\label{subsection: PHI}

As defined in Section \ref{section: introduction}, $M$ is a thickened nilmanifold, 
$\mathrm{Diff}_\partial^k(M)$ is the collection of $C^k$ diffeomorphisms on $M$ preserving each boundary $M_0$ and $M_1$ respectively, 
$\mathrm{PHI}^k(M)$ is the collection of partially hyperbolic diffeomorphisms in $\mathrm{Diff}_\partial^k(M)$ with interval central leaves, 
and $\mathrm{PHI}_-^k(M)$ is the collection of diffeomorphisms in $\mathrm{PHI}^k(M)$ whose boundary SRB measures admit negative central Lyapunov exponents.

For $F\in\mathrm{PHI}^2(M)$, the dynamical foliations $\mathcal{F}^\sigma$ are H\"older continuous with $C^{1+}$ leaves, 
for $\sigma = s,c,u,cs,cu$. 
A subset of $M$ is \textit{$\sigma$-saturated}, if it is a union of $\sigma$-leaves.
There is a well-known result on the regularity of holonomy maps in \cite{PSW97}, 
which we state in our setting as follows. 
Note that $F\in\mathrm{PHI}^2(M)$ is dynamically coherent and \textit{center bunched}, 
in the sense that for any $x\in M$,
$$
\|DF|_{E^s(x)}\| < \frac{m(DF|_{E^c(x)})}{\|DF|_{E^c(x)}\|} 
\quad\text{ and }\quad
m(DF|_{E^u(x)}) > \frac{\|DF|_{E^c(x)}\|}{m(DF|_{E^c(x)})}.
$$

\begin{lemma}[Theorem B, \cite{PSW97}]\label{lemma: C1-holonomy}
    For $F\in\mathrm{PHI}^2(M)$, 
    the holonomy map along the stable foliation (resp. the unstable foliation) restricted to a center-stable leaf (resp. a center-unstable leaf) is a $C^1$ diffeomorphism. 
    Moreover, the derivative of the holonomy map is uniformly continuous with respect to the base points.
\end{lemma}

The following result on the estimate of central twist is useful in our proof.

\begin{lemma}\label{lemma: central twist}
    Assume that $F\in\mathrm{PHI}^2(M)$ and $p\in M_0$ is a fixed point of $F_0$ with $\lambda^c(p) < 0$. 
    Fix $r\in \mathcal{F}^u(p)$ 
    and let $\mathrm{Hol}_p^r: \mathcal{F}^{cs}_{loc}(p) \to \mathcal{F}^{cs}_{loc}(r)$ be the holonomy map along the unstable foliation. 
    For $x_0\in\mathcal{F}^{cs}_{loc}(p)$, 
    denote $x_n = F^n(x_0)$ for every $n\in\mathbb{N}$, 
    and define
    $$
    y_n    := \mathcal{F}^s_{loc}(x_n)\pitchfork\mathcal{F}^c_{loc}(p),
    \ z_n  := \mathrm{Hol}_p^r(x_n),
    \ w_n  := \mathrm{Hol}_p^r(y_n),
    \ w'_n := \mathcal{F}^s_{loc}(z_n)\pitchfork\mathcal{F}^c_{loc}(r).
    $$
    Then we have the following estimate:
    $$
    d_c(w'_n,\, w_n) \leq Ce^{n\lambda_p},\ \forall\, n\in\mathbb{N}.
    $$
    Here $C\geq 1$ is a constant depending on $p$ and $r$, and
    $$
    \lambda_p 
    := \frac{\lambda^u(p) - \lambda^c(p)}{\lambda^u(p) - \lambda^s(p)}\lambda^s(p) 
    < \lambda^c(p) < 0,
    $$
    where
    $$
    \ln\left\|DF|_{E^s(p)}\right\| 
    =: 
    \lambda^s(p) < \lambda^c(p) <  \lambda^u(p) 
    := 
    \ln m\left(DF|_{E^u(p)}\right).
    $$
\end{lemma}

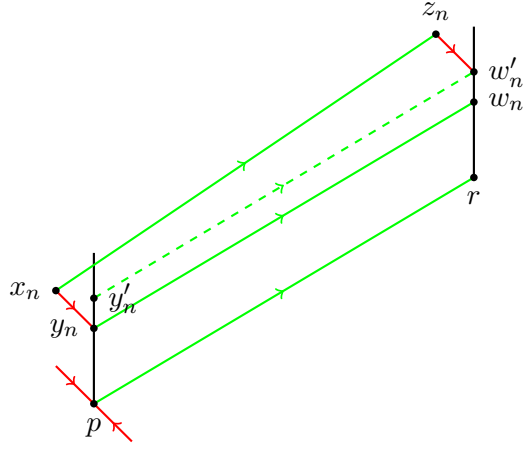
\begin{figure}[htbp]
    \centering
    \begin{tikzpicture}[mid arrow/.style = {decoration = {markings, mark = at position 0.5 with {\arrow{>}}}, postaction = {decorate}}]
        \draw [thick, mid arrow, color = red] (-0.5,0.5)--(0,0);
        \draw [thick, mid arrow, color = red] (0.5,-0.5)--(0,0);
        \draw [thick, mid arrow, color = green] (0,0)--(5,3);
        \draw [thick] (0,0)--(0,2);
        \draw [thick] (5,3)--(5,5);
        \node [circle, fill, inner sep = 1pt, label = below: $p$] at (0,0) {};
        \node [circle, fill, inner sep = 1pt, label = below: $r$] at (5,3) {};

        \draw [thick, mid arrow, color = red] (-0.5,1.5)--(0,1);
        \draw [thick, mid arrow, color = green] (-0.5,1.5)--(4.5,4.9);
        \draw [thick, mid arrow, color = green] (0,1)--(5,4);
        \draw [thick, mid arrow, color = red] (4.5,4.9)--(5,4.4);
        \draw [dashed, thick, mid arrow, color = green] (0,1.4)--(5,4.4);
        
        \node [circle, fill, inner sep = 1pt, label = left: $x_n$] at (-0.5,1.5) {};
        \node [circle, fill, inner sep = 1pt, label = left: $y_n$] at (0,1) {};
        \node [circle, fill, inner sep = 1pt, label = above: $z_n$] at (4.5,4.9) {};
        \node [circle, fill, inner sep = 1pt, label = right: $w_n$] at (5,4) {};
        \node [circle, fill, inner sep = 1pt, label = right: $w'_n$] at (5,4.4) {};
        \node [circle, fill, inner sep = 1pt, label = right: $y'_n$] at (0,1.4) {};
    \end{tikzpicture}
    \caption{Central Twist}
    \label{figure: central twist}
\end{figure}

\begin{proof}
    Take $y'_n := (\mathrm{Hol}_p^r)^{-1}(w'_n)$, see Figure \ref{figure: central twist}. 
    By Lemma \ref{lemma: C1-holonomy}, the holonomy map $\mathrm{Hol}_p^r$ restricted to $\mathcal{F}^c_{loc}(p)$ is a $C^1$ diffeomorphism. 
    Hence, it suffices to estimate $d_c(y'_n,\ y_n)$. 
    First notice that
    $$
    d_c\left(F^{-N}(y'_n),\, F^{-N}(y_n)\right) \leq 2d\left(F^{-N}(y'_n),\, F^{-N}(y_n)\right)
    $$
    when $n - N$ is sufficiently large. 
    Therefore, there exists a constant $C\geq1$ such that
    \begin{align*}
        &d_c\left(F^{-N}(y'_n),\, F^{-N}(y_n)\right)\\
        \leq& 2\left[d_u\left(F^{-N}(y'_n),\, F^{-N}(w'_n)\right) + d_s\left(F^{-N}(w'_n),\, F^{-N}(z_n)\right)\right.\\
        &+ \left.d_u\left(F^{-N}(z_n),\, F^{-N}(x_n)\right) + d_s\left(F^{-N}(x_n),\, F^{-N}(y_n)\right)\right]\\
        \leq &C\left[e^{-N\lambda^u(p)} + e^{(n - N)\lambda^s(p)}\right].
    \end{align*}
    Now let $-N\lambda^u(p) \approx (n - N)\lambda^s(p)$, we take
    $$
    N = \left[\frac{-\lambda^s(p)}{\lambda^u(p) - \lambda^s(p)}n\right] + 1.
    $$
    It follows that when $n$ is sufficiently large, we have
    \begin{align*}
        d_c(y'_n,\, y_n) &\leq Ce^{N\lambda^c(p)}d_c\left(F^{-N}(y'_n),\, F^{-N}(y_n)\right)\\
        &\leq 2C^2e^{N\lambda^c(p) + (n - N)\lambda^s(p)}\\
        &\leq 4C^2e^{n\lambda_p},
    \end{align*}
    where
    $$
    \lambda_p 
    = \lim_{n\to+\infty}\left(\frac{N}{n}\lambda^c(p) + \frac{n - N}{n}\lambda^s(p)\right) 
    = \frac{\lambda^u(p) - \lambda^c(p)}{\lambda^u(p) - \lambda^s(p)}\lambda^s(p) 
    < \lambda^c(p) < 0.
    $$
    This completes the proof of Lemma \ref{lemma: central twist}.
\end{proof}

\subsection{Asymptotically rational independence of Birkhoff sums}
Here we collect some results related to Birkhoff sums in the uniformly hyperbolic setting.

First, there is a concept of asymptotically rational independence introduced in \cite{GS19}. 
For real numbers $a,b\in\mathbb{R}$, 
define
$$
(a,\ b) := \inf\left\{\, |ka + lb|: k,l\in\mathbb{Z},\ ka + lb\neq 0 \,\right\}.
$$
Particularly, this notation $(a,\ b)$ coincides with $\gcd(a,b)$ when $a,b\in\mathbb{Z}$.
Recall that two real numbers $a,b$ are \textit{rationally independent} if and only if $(a,\ b) = 0$. 
Inspired by this fact, the asymptotically rational independence is defined as follows.

\begin{definition}\label{definition: ARI}
    Let $b\in\mathbb{R}$. 
    A sequence $\{a_m\}\subseteq\mathbb{R}$ is \textit{asymptotically rationally independent} of $b$, 
    if $(a_m,\ b) \to 0$ as $m\to +\infty$.
\end{definition}

To understand the asymptotically rational independence, we state the following results.

\begin{lemma}\label{lemma: ARI}
    If $a,b\in\mathbb{R}$ and $b\neq 0$, 
    then $(a,\ b) < \varepsilon$ for some $\varepsilon > 0$ if and only if 
    each orbit of $T_a(x) := x + a\ (\mathrm{mod}\ b\mathbb{Z})$ is $\varepsilon$-dense in $\mathbb{R}/(b\mathbb{Z})$.
    In particular, if $(a,\ b) > 0$, 
    then $(a,\ b)=|b|/q$ for the reduced fraction $a/b=p/q$ with $p\in\mathbb{Z},\ q\in\mathbb{N},\ (p,\ q) = 1$.
\end{lemma}

\begin{proof}
    The conclusion follows directly from the definition of $(a,\ b)$.
\end{proof}

\begin{lemma}\label{lemma: small distance}
    If $a_1,a_2,b\in\mathbb{R}$ and $b\neq 0$, 
    then $0 < |a_1 - a_2| < \varepsilon\ (\mathrm{mod}\ b\mathbb{Z})$ implies that 
    there exists $a\in\{a_1,a_2\}$ such that $(a,\ b) < \sqrt{|b|\,\varepsilon}$.
\end{lemma}

\begin{proof}
    Given $b\neq 0$, notice that
    \begin{align*}
        \left\{ a\ (\mathrm{mod}\ b\mathbb{Z}): (a,\ b)\geq {\sqrt{|b|\,\varepsilon}} \right\} = 
        \left\{
        \frac{k}{n}b\ (\mathrm{mod}\ b\mathbb{Z}): 
        n\in\mathbb{Z}^+,\ n \leq {\sqrt{\frac{|b|}{\varepsilon}}},\ k\in\mathbb{Z},\ (k,\ n) = 1\ 
        \right\}
    \end{align*}
    is a finite set,
    in which for any $a_1,a_2$ with $a_1\neq a_2 \ {(\mathrm{mod}\ b\mathbb{Z})}$, 
    we have $|a_1 - a_2| \geq \varepsilon\ (\mathrm{mod}\ b\mathbb{Z})$. 
    In fact, taking the reduced fractions
    \[
    \frac{a_i}{b}=\frac{p_i}{q_i} (\mathrm{mod}\ \mathbb{Z}), 
    \text{ with } p_i\in\mathbb{Z},\ q_i\in\mathbb{N},\ (p_i,\ q_i) = 1,
    \] 
    we have $p_1q_2-p_2q_1-kq_1q_2\neq 0$ for any $k\in\mathbb{Z}$, hence
    \[
    \left| \frac{p_1}{q_1}-\frac{p_2}{q_2}-k \right|=\frac{|p_1q_2-p_2q_1-kq_1q_2|}{q_1q_2}\geq \frac{1}{q_1q_2},
    \]
    and then
    \begin{align*}
        |a_1 - a_2|\ (\mathrm{mod}\ b\mathbb{Z})
        &= \inf_k\left\{\left| \frac{p_1}{q_1}-\frac{p_2}{q_2}-k \right|\cdot|b|\right\} \\
        &\geq  \frac{|b|}{q_1q_2}=\frac{|b|/q_1\cdot |b|/q_2}{|b|}=\frac{(a_1,\ b)\cdot (a_2,\ b)}{|b|} \\
        &\geq  \frac{\sqrt{|b|\,\varepsilon}\cdot\sqrt{|b|\,\varepsilon}}{|b|}=\varepsilon.
    \end{align*}
    Therefore, $0 < |a_1 - a_2| < \varepsilon\ (\mathrm{mod}\ b\mathbb{Z})$ 
    implies that at least one of them does not belong to the finite set.
\end{proof}

Now we introduce the Birkhoff sums. 
Let $A: N \to N$ be an Anosov diffeomorphism on a nilmanifold 
and $\varphi: N \to \mathbb{R}$ be a H\"older continuous function. 
Denote by
$$
S_\varphi A(p) := \sum_{i = 0}^{\pi(p) - 1}{\varphi(A^i(p))}
$$
the \textit{Birkhoff sum} of the function $\varphi$ along a periodic orbit of $p\in\mathrm{Per}(A)$ with period $\pi(p)$.

The following \textbf{Positive Liv\v sic Theorem} 
(see \cite{LT03}) or \textbf{Ma\~n\'e-Conze-Guivarc'h-Bousch Lemma} (see \cite{Bou01,MCGB11}) 
is a classical method to obtain the information of the whole space from that of periodic orbits.

\begin{lemma}[Positive Liv\v{s}ic Theorem]\label{lemma: positive Livsic}
	Let $A: N \to N$ be an Anosov diffeomorphism on a nilmanifold and $\varphi: N \to \mathbb{R}$ be a H\"older continuous function. 
    If
	$$
    S_\varphi A(p)\geq 0,\ \forall\, p\in \mathrm{Per}(A),
    $$
    then there exists a H\"older continuous function $\psi: N \to \mathbb{R}$ such that $\varphi\geq \psi - \psi\circ A$.
\end{lemma}

For $F\in\mathrm{PHI}^1(M)$, we have an Anosov diffeomorphism $F_0: M_0 \to M_0$. 
Define
$$
\varphi: M_0 \to \mathbb{R}
\text{ and }
\varphi(x) := \ln\|DF|_{E^c(x)}\|.
$$
Then for a periodic point $p\in \mathrm{Per}(F_0)$ with period $\pi(p)$, 
the corresponding Birkhoff sum $S_\varphi F_0(p)$ is exactly \textit{total central Lyapunov exponent} $\pi(p)\lambda^c(p)$. 
Hence, the Birkhoff sum is regarded as an observation function involving the information of central Lyapunov exponents. 

In particular, we have the following corollary of Lemma \ref{lemma: positive Livsic}.

\begin{corollary}\label{corollary: central Lyapunov exponent}
    For $F\in\mathrm{PHI}^1(M)$, if $\lambda^c(p) \leq 0$ for every $p\in\mathrm{Per}(F_0)$, 
    then $\lambda^c(x) \leq 0$ whenever $\lambda^c(x)$ exists.
\end{corollary}

There is the following rigidity result on the distribution of all Birkhoff sums.

\begin{lemma}[Theorem 1.1, \cite{GSX22}]\label{lemma: dense distribution}
    Let $A: N \to N$ be an Anosov diffeomorphism on a nilmanifold and $\varphi: N \to \mathbb{R}$ be a H\"older continuous function.
    If there are two periodic points $p_0,q_0\in \mathrm{Per}(A)$ satisfying 	
	$$S_\varphi A(p_0) < 0 < S_\varphi A(q_0),$$
	Then the set $\left\{S_\varphi A(p): p\in \mathrm{Per}(A)\right\}$ is dense in $\mathbb{R}$.
\end{lemma}

\subsection{Analysis in the one-dimensional dynamics}
The following quantitative result is prepared for analyzing the one-dimensional center dynamics connected by holonomy map. 
The proof is technical and may be skipped on a first reading.

\begin{proposition}\label{proposition: center intersection}
	Let $f,g\in\mathrm{Diff}_\partial^{1+}([0,1])$ be diffeomorphisms with exactly two fixed points $0,1$, 
    where $0$ is a hyperbolic sink; 
    and $h\in\mathrm{Diff}_\partial^1([0,1])$. 
    Then there exist $\varepsilon_0 > 0$ and an increasing function $L(\varepsilon): (0,\varepsilon_0) \to (0,1)$ 
    satisfying $L(x)\to 0$ as $x\to 0^+$ and depending on $\beta := g'(0)$, such that the following conclusion holds: 
    If $\alpha := f'(0)$ satisfies
    $$(\ln \alpha,\ \ln\beta) < \varepsilon,$$
    then for any two closed intervals $I,J \subseteq (0,1)$ with $|J| > L(\varepsilon)$, 
    there exist two strictly increasing subsequences $\{k_n\},\{l_n\}$ of $\mathbb{N}$, such that
    \begin{align*}
        &\sup_{n\in\mathbb{N}}\big\{\,\big|k_n\ln\alpha - l_n\ln\beta\big|\,\big\} < +\infty;\\
        &h\circ f^{k_n}(I) \cap g^{l_n}(J) \neq \varnothing,\ \forall\, n\in\mathbb{N}.
    \end{align*}
	Moreover, there exists a constant $c > 0$ such that
	$$
    \left|h\circ f^{k_n}(I)\cap g^{l_n}(J)\right|\geq c\cdot\max\left\{e^{k_n\ln \alpha},\, e^{l_n\ln \beta}\right\},
    \ \forall\, n\in\mathbb{N}.
    $$
\end{proposition}

\begin{proof}
    First, we have the following expansions:
    \begin{flalign*}
         f(x) &= \alpha x + R_f(x), \              0 < \alpha = f'(0) < 1,    \              R_f(x) = o(x)\ ({x\to 0^+});\\
         g(x) &= \beta  x + R_g(x), \ \,           0 <  \beta = g'(0) < 1,    \ \,           R_g(x) = o(x)\ ({x\to 0^+});\\
         h(x) &= \gamma x + R_h(x), \quad ~\ ~\, ~\,   \gamma = h'(0) > 0,    \quad \ \      R_h(x) = o(x)\ ({x\to 0^+}).
    \end{flalign*}
    For any $k,l\in\mathbb{N}$, we also have
    $$g^{-l}\circ h\circ f^k(x) = \beta^{-l}\alpha^kh(x) + R_{k,l}(x),\ R_{k,l}(x) = o(x)\ ({x\to 0^+}).$$
    Hence, for any $\varepsilon > 0$, we can take $0 < \delta < 1$ sufficiently small, such that
    \begin{align*}
        &\max\{~|R_f(x)|,\, |R_g(x)|,\, |R_h(x)|~\} < \varepsilon x,\ \forall\, x\in [0,\, \max\{\delta, h(\delta)\}];\\
        &\max\{~|R'_f(x)/\alpha|,\, |R'_g(x)/\beta|,\, |R'_h(x)/\gamma|~\} < \frac{1}{2},\ \forall\, x\in[0,\,\max\{\delta,h(\delta)\}].
    \end{align*}
    Since $(\ln\alpha,\ \ln\beta) < \varepsilon$, 
    there exist $K^*$, $L^*\in\mathbb{Z}^+$ depending on $\beta$, 
    such that for any $\eta \in (\beta^2, 1)$, 
    there exist $1\leq K\leq K^*$, $1\leq L\leq L^*$ such that
    $$
    K\ln\alpha - L\ln\beta \in (\ln\eta,\, \ln\eta + \varepsilon).
    $$
    Therefore, by taking $\delta$ smaller, we may further require that
    $$
    \max\{|R_{k,l}(x)|: 1\leq k \leq K^*,\ 1\leq l \leq L^*\} < \varepsilon h(x),\ \forall\, x\in[0, \delta].
    $$
    It follows that
    $$g^{-L}\circ h \circ f^K(x) \in \big((\eta - \varepsilon) h(x),\, (e^\varepsilon\eta + \varepsilon) h(x)\big).$$
    
    For $x\in[0, \delta]$, define $I_x := (g(h(x)),\, h(x)]$. Then we have the following conclusion.
    
    \begin{claim}\label{claim: local intersection}
        For any $x\in [0,\delta]$ and any closed interval $J\subseteq I_x$ with
        $$\frac{|J|}{|I_x|} > \frac{3\varepsilon}{1 - \beta - \varepsilon},$$
        there exist $1\leq K \leq K^*$ and $1\leq L \leq L^*$ such that
        $$g^{-L}\circ h \circ f^K(x) \in J.$$
    \end{claim}
    \begin{proof}[Proof of Claim \ref{claim: local intersection}]
        Notice that $|I_x| > (1 - \beta - \varepsilon)h(x)$, and hence $|J| > 3\varepsilon h(x)$. 
        It follows that there exists $\eta\in(\beta^2, e^{-\varepsilon})$ such that
        $$\big((\eta - \varepsilon) h(x),\, (e^\varepsilon\eta + \varepsilon) h(x)\big) \subseteq J,$$
        because the error $((e^\varepsilon - 1)\eta + 2\varepsilon)h(x) < 3\varepsilon h(x)$. 
        It follows that there exist $1\leq K \leq K^*$, $1\leq L \leq L^*$ such that
        $$g^{-L}\circ h \circ f^K(x) \in \big((\eta - \varepsilon) h(x),\, (e^\varepsilon\eta + \varepsilon) h(x)\big) \subseteq J.$$
        This completes the proof of Claim \ref{claim: local intersection}.
    \end{proof}

        For any point $y\in (0,1)$ and closed interval $J\subseteq (0,1)$, 
        define the fundamental domains $\Omega_i := (g^{i + 1}(y),\, g^i(y)]$, 
        and denote $J^* := \bigcup_{i\in\mathbb{Z}}g^i(J)$ 
        and $J_i := J^*\cap \Omega_i$ for any $i\in\mathbb{Z}$.
        Then we have the following conclusion.
    
    \begin{claim}\label{claim: uniform proportion}
        There exists a constant $0 < \rho < 1$ such that
        $$
        \rho\frac{|J_i|}{|\Omega_i|} \leq \frac{|J_j|}{|\Omega_j|} \leq \rho^{-1}\frac{|J_i|}{|\Omega_i|} 
        \ ~ \text{ whenever }\ ~ \Omega_i, \Omega_j \subseteq [0, h(\delta)].
        $$
        In particular, there exists a constant $c > 0$ such that $|g^i(J)| \geq c\beta^i,\ \forall\, i\in\mathbb{N}$.
    \end{claim}
    \begin{proof}[Proof of Claim \ref{claim: uniform proportion}]
        Note that $J_i$ does not change when we replace $J$ by $J_0$. 
        Since $J_0$ has at most two components, by considering each component respectively, 
        we may assume at first that $J = [a,b]\subseteq \overline{\Omega}_0$. 
        Without loss of generality, assume that $j > i$. Then we have
        \begin{align*}
            \frac{|J_j|}{|\Omega_j|}
            &= \frac{g^j(b) - g^j(a)}{g^j(y) - g^{j + 1}(y)}\\
            &= \frac{g^i(b) - g^i(a)}{g^i(y) - g^{i + 1}(y)}\cdot\frac{(g^{j - i})'(\xi_1)}{(g^{j - i})'(\xi_2)}\\
            &= \frac{|J_i|}{|\Omega_i|}\cdot\exp\left[\sum_{k = 0}^{j - i - 1}\left(\ln g'(g^k(\xi_1)) - \ln g'(g^k(\xi_2))\right)\right],
        \end{align*}
        where $\xi_1\in J_i$ and $\xi_2\in\Omega_i$, and hence
        \begin{align*}
            &\quad\left| \sum_{k = 0}^{j - i - 1}\left(\ln g'(g^k(\xi_1)) - \ln g'(g^k(\xi_2))\right) \right|\\
            &\leq        \sum_{k = 0}^{j - i - 1}\left|\ln g'(g^k(\xi_1)) - \ln g'(g^k(\xi_2))\right|\\
            &\leq        \sum_{k = 0}^{j - i - 1}C|g^k(\xi_1) - g^k(\xi_2)|^\theta\\
            &\leq        \sum_{k = 0}^{j - i - 1}C^2\beta^{k\theta}|\xi_1 - \xi_2|^\theta \leq \frac{C^2}{1 - \beta^\theta}.
        \end{align*}
        Here $0 < \theta < 1$ is the H\"older exponent of $\ln g'$ in $[0,h(\delta)]$, 
        and $C \geq 1$ is a constant depending on $g$.
        Therefore, we can take
        $$\rho = \exp \left[ - \frac{C^2}{1 - \beta^\theta}\right].$$
        This completes the proof of Claim \ref{claim: uniform proportion}.
    \end{proof}
    Consider the family of fundamental domains
    $$
    \mathcal{D} := \{I_t := (g(t), t]: t\in D_0 := \big(g(h(\delta)),h(\delta)\big]\},\ D := \sup_{t\in D_0}|I_t|.
    $$
    \begin{claim}\label{claim: fundamental intersection}
        There exist $\varepsilon_0 > 0$ and an increasing function $G: (0,\varepsilon_0) \to (0, 1)$ 
        satisfying $G(x)\to 0$ as $x\to 0^+$ and $r = r(s, t): D_0\times D_0 \to \{-1,\, 0,\, 1\}$, 
        such that
		$$|g^r(J_s(\varepsilon))\cap I_t| \geq \varepsilon,\ \forall\, s,t\in D_0,$$
		where $J_s(\varepsilon) := [s - G(\varepsilon),\, s]$.
    \end{claim}
    \begin{proof}[Proof of Claim \ref{claim: fundamental intersection}]
        Fix any $s\in D_0$. Considering the relative position of $s$ and $I_t$, we need the following sufficient conditions:
			\begin{flalign*}
				&\qquad\qquad t - g^{-1}(s - G(\varepsilon)) \geq \varepsilon,& &\text{when } g(t) < s < g(t) + \varepsilon;&\\
				&\qquad\qquad G(\varepsilon) \geq\varepsilon,                 & &\text{when } g(t) + \varepsilon \leq s \leq t;&\\
				&\qquad\qquad t - (s - G(\varepsilon)) \geq \varepsilon,      & &\text{when } t < s < g^{-1}(g(t) + \varepsilon);&\\
				&\qquad\qquad g(s) - g(s - G(\varepsilon)) \geq \varepsilon,  & &\text{when } s \geq g^{-1}(g(t) + \varepsilon).&
			\end{flalign*}
			Finally, we take
			$$
            G(\varepsilon) = \sup_{u \in D_0}
            \left\{
            g(u) + \varepsilon - g(u - \varepsilon),
            \ \varepsilon,
            \ g^{-1}(g(u) + \varepsilon) - u + \varepsilon,
            \ u - g^{-1}(g(u) - \varepsilon)
            \right\}.
            $$
			One can check that $G(\cdot)$ is increasing and $G(x)\to 0$ as $x\to 0^+$. 
            Note that we take $\varepsilon_0 > 0$ sufficiently small so that 
            when $0 < \varepsilon < \varepsilon_0$, the above discussions make sense.
    \end{proof}
    Now we prove Proposition \ref{proposition: center intersection}. 
    Define
	$$
    L(\varepsilon) 
    := \sup\left\{\left|g^{-l}\left[J_s\left(\rho^{-1}\frac{3\varepsilon}{1 - \beta - \varepsilon}D\right)\right]\right|: 
    s\in D_0,\ l\in\mathbb{Z}\right\}.
    $$
	Note that we take $\varepsilon_0 > 0$ sufficiently small so that 
    when $0 < \varepsilon < \varepsilon_0$, the function $L(\cdot)$ is well-defined. 
    One can check that $L(\cdot)$ is increasing and $L(x)\to 0$ as $x\to 0^+$.
    
    We claim that $|J| > 3L(\varepsilon)$ will be sufficient.
	
	In fact, for any point $x\in(0, 1)$ and closed interval $J\subseteq (0, 1)$ with $|J| > L(\varepsilon)$, 
    there exist $k^*\in\mathbb{N}$ and $l^*\in\mathbb{Z}$ such that 
    $f^{k^*}(x)\in(0, \delta)$ and $g^{l^*}(J) = [s - |g^{l^*}(J)|,\, s]$ for some $s\in D_0$. 
    Moreover,
	$$
    |g^{l^*}(J)| 
    > \left|J_s\left(\rho^{-1}\frac{3\varepsilon}{1 - \beta - \varepsilon}D\right)\right| 
    = G\left(\rho^{-1}\frac{3\varepsilon}{1 - \beta - \varepsilon}D\right).
    $$
		
	By Claim \ref{claim: fundamental intersection}, we have
	$$
    \frac{|g^{l^* + r}(J)\cap I_t|}{|I_t|} 
    > \rho^{-1}\frac{3\varepsilon}{1 - \beta - \varepsilon},
    \ \forall\,  t\in D_0,
    $$
	where $r\in\{-1,\, 0,\, 1\}$ depends on $J$ and $t\in D_0$. 
    In particular, for each $n\in\mathbb{N}$ large enough, 
    consider the fundamental domain
	$$
    I_{f^{k^* + n}(x)} = \Big(g(h(f^{k^* + n}(x))),\, h(f^{k^* + n}(x))\Big],
    $$
	and let $\widehat{l}_n\in\mathbb{N}$ be the unique integer such that 
    $$
    g^{-\widehat{l}_n}(I_{f^{k^* + n}(x)}) = I_t,\ \exists\ t\in D_0.
    $$
	Then we have
	$$
    \frac{\big|g^{l^* + r_n}(J)\cap g^{-\widehat{l}_n}(I_{f^{k^* + n}(x)})\big|}{\big|g^{-\widehat{l}_n}(I_{f^{k^* + n}(x)})\big|} 
    > \rho^{-1}\frac{3\varepsilon}{1 - \beta - \varepsilon},
    $$
	where $r_n\in\{-1,\, 0,\, 1\}$ depends on $x$ and $J$.
	
	By Claim \ref{claim: uniform proportion}, we have
	$$
    \frac{\big|g^{l^* + r_n + \widehat{l}_n}(J)\cap I_{f^{k^* + n}(x)}\big|}{\big|I_{f^{k^* + n}(x)}\big|} 
    > \frac{3\varepsilon}{1 - \beta - \varepsilon}.
    $$
	
	By Claim \ref{claim: local intersection}, there exist $1\leq K_n \leq K^*$ and $1\leq L_n \leq L^*$ such that
	$$g^{-L_n}\circ h\circ f^{K_n}(f^{k^* + n}(x)) \in g^{l^* + r_n + \widehat{l}_n}(J).$$
	
	Finally, we take $k_n = n + k^* + K_n$ and $l_n = \widehat{l}_n + l^* + r_n + L_n$, and then we have
	$$g^{-l_n}\circ h\circ f^{k_n}(x) \in J.$$
	Since $k_n \to +\infty$ and $l_n \to +\infty$ as $n \to +\infty$, 
    by taking a subsequence if necessary, 
    we can require that the sequences $\{k_n\}$ and $\{l_n\}$ are positive and strictly increasing.

    It follows that for any two closed intervals $I,J\subseteq(0, 1)$ with $|J| > 3L(\varepsilon)$, 
    there exist strictly increasing subsequences $\{k_n\}$, $\{l_n\}$ of $\mathbb{N}$, such that
    $$h\circ f^{k_n}(I)\cap g^{l_n}(J')\not=\varnothing,$$
    where $J'$ is the middle third of $J$. 
    Then, we have the following two cases.

    \noindent\textbf{Case 1. $h\circ f^{k_n}(I)$ is totally contained in $g^{l_n}(J)$.} 
    Applying Claim \ref{claim: uniform proportion} to $f$, we have
    $$
    \big|h\circ f^{k_n}(I)\cap g^{l_n}(J)\big| = \big|h\circ f^{k_n}(I)\big| 
    \geq c\alpha^{k_n},
    $$
    for some constant $c > 0$.

    \noindent\textbf{Case 2. $h\circ f^{k_n}(I)$ is not totally contained in $g^{l_n}(J)$.} 
    Assume that $J = [a,b]$ and $J' = [a',b']$, 
    then $h\circ f^{k_n}(I)$ contains $g^{l_n}([a, a'])$ or $g^{l_n}([b',b])$. 
    Applying Claim \ref{claim: uniform proportion}, we have
    $$
    \big|h\circ f^{k_n}(I)\cap g^{l_n}(J)\big| 
    \geq \min\left\{\big|g^{l_n}([a,a'])\big|,\, \big|g^{l_n}([b',b])\big|\right\} \geq c\beta^{l_n},
    $$
    for some constant $c > 0$.
    
    Recall that by the choice of $\widehat{l}_n$ we have
    $$h(f^{k^* + n}(x))\in \Big(g^{\widehat{l}_n + 1}(h(\delta)),\, g^{\widehat{l}_n}(h(\delta))\Big].$$
    Hence, there exists a constant $C \geq 1$ such that
    \begin{align*}
        &C\gamma\alpha^{k^* + n}x \geq C^{-1}\beta^{\widehat{l}_n + 1}h(\delta);\\
        &C^{-1}\gamma\alpha^{k^* + n}x \leq C\beta^{\widehat{l}_n}h(\delta).
    \end{align*}
    And we conclude that
    $$
    \sup_{n\in\mathbb{N}}\big\{\,\big|k_n\ln\alpha - l_n\ln\beta\big|\,\big\} < +\infty, 
    \text{ and } 
    \lim_{n\to+\infty}\frac{l_n}{k_n} = \frac{\ln\alpha}{\ln\beta}.
    $$
    As a result, there exists a constant $c > 0$ such that
    $$
    |h\circ f^{k_n}(I)\cap g^{l_n}(J)| 
    \geq c\cdot\max\left\{e^{k_n\ln\alpha},\, e^{l_n\ln\beta}\right\},\ \forall\, n\in\mathbb{N}.
    $$
    This completes the proof of Proposition \ref{proposition: center intersection}.
\end{proof}

\section{Boundary Interconnection Implies Topological Transitivity}\label{section: BI to TT}
In this section, we prove that boundary interconnection implies topological transitivity.

\begin{theorem}\label{theorem: BI to TT}
	Assume that $F\in\mathrm{PHI}^2(M)$. If $F$ is boundary interconnected, then $F$ is topologically transitive.
\end{theorem}

First, we characterize the stable set of a periodic point with a uniformly contracting center.

\begin{lemma}\label{lemma: W^s and W^u}
    Let $p\in\mathrm{Per}(F_0)$ be a periodic point with period $\pi(p)$ and $\lambda^c(p) < 0$. 
    Assume that $\gamma_p: [0, 1] \to \mathcal{F}^c(p)$ is 
    a $C^{1+}$ parameterization of the central leaf of $p$ with $\gamma_p(0) = p$. 
    Then there exists $0 < s_p \leq 1$ satisfying
    $$
    I_p^c := W^s(p)\cap\mathcal{F}^c(p) = \gamma_p([0, s_p))
    \, \text{ and }\, 
    W^s(p) = \bigcup_{x\in I_p^c}\mathcal{F}^s(x).
    $$
    In particular, 
    define \(\ell_p:[0,1]\to[0,s_p]\) and \(\ell_p(t):=s_p t\), then
    $\gamma_p^s := \gamma_p\circ {\ell_p} : [0, 1] \to \bar{I}_p^c$ is a $C^{1+}$ parameterization, 
    and 
    $$(\gamma_p^s)^{-1}\circ F^{\pi(p)} \circ \gamma_p^s \in\mathrm{Diff}_\partial^{1+}([0, 1])$$
    is a diffeomorphism with exactly two fixed points $0,1$, where $0$ is a hyperbolic sink.
\end{lemma}

\begin{proof}[Proof of Lemma \ref{lemma: W^s and W^u}]
    Notice that $f_p := (\gamma_p)^{-1}\circ F^{\pi(p)}\circ \gamma_p \in\mathrm{Diff}_\partial^{1+}([0, 1])$. 
    Since $\lambda^c(p) < 0$, we have $0 < f'_p(0) < 1$, 
    and therefore $f_p$ has a smallest positive fixed point $0 < s_p \leq 1$. 
    It follows that ${\ell_p}^{-1}\circ f_p\circ {\ell_p}\in \mathrm{Diff}_\partial^{1+}([0, 1])$ is a diffeomorphism with exactly two fixed points $0,1$, 
    where $0$ is a hyperbolic sink.

    By the above argument, 
    $I_p^c := W^s(p)\cap\mathcal{F}^c(p) = \gamma_p([0, s_p))$. 
    Therefore we have
    $$W^s(p) \supseteq \bigcup_{x\in I_p^c}\mathcal{F}^s(x).$$
    Now take $y\in W^s({p})$. 
    Then we have $\pi_0(y)\in\mathcal{F}^s(\pi_0({p}))$. 
    Recall that $\pi_0: M \to M_0$ is the natural projection along central leaves. 
    It follows that $y\in\mathcal{F}^{cs}({p})$. 
    Assume that $y\in\mathcal{F}^s(z)$ for some $z\in\mathcal{F}^c({p})$. 
    Then $z\in W^s(p)$ and hence $z\in I_p^c$, since $W^s(p)$ is $s$-saturated. 
    This completes the proof of Lemma \ref{lemma: W^s and W^u}.
\end{proof}

Note that similar conclusions hold for periodic points of both $F_0$ and $F_1$: 
The stable set (resp. unstable set) is characterized by the uniformly contracting center (resp. uniformly expanding center) 
when the central Lyapunov exponent is negative (resp. positive). 
For simplicity, we omit the statement for the other case.

Now we prove Theorem \ref{theorem: BI to TT}.

\begin{proof}[Proof of Theorem \ref{theorem: BI to TT}]
    It suffices to show that for any two non-empty open subsets $U, V\subseteq M$, 
    there exists $N\in\mathbb{N}$ such that $F^N(U)\cap V\neq\varnothing$. 
    Equivalently, we only need to show that $U^*\cap V^*\neq\varnothing$, where
    $$U^* := \bigcup_{n\in\mathbb{N}}F^n(U) \text{ and } V^* := \bigcup_{n\in\mathbb{N}}F^{-n}(V).$$
    Note that $F(U^*)\subseteq U^*$, hence $U^*$ is $u$-saturated; $F^{-1}(V^*)\subseteq V^*$ and $V^*$ is $s$-saturated.
    
    Since $F$ is boundary interconnected, 
    there exist two pairs of periodic points $p_i,q_i\in M_i$ such that
    \begin{align*}
        \lambda^c(p_0) < 0 < \lambda^c(p_1) &\text{ and } W^s(p_0)\pitchfork W^u(p_1)\neq\varnothing;\\
        \lambda^c(q_1) < 0 < \lambda^c(q_0) &\text{ and } W^s(q_1)\pitchfork W^u(q_0)\neq\varnothing.
    \end{align*}
    Without loss of generality, we assume that the periodic points $p_i$ and $q_i$ are fixed points. 
    Define $\varphi: M \to \mathbb{R}$ by
    $$\varphi(x) := \ln\|DF|_{E^c(x)}\|.$$
    Then $\varphi$ is H\"older continuous, and there exist constants $C_\varphi \geq 1$ and $0 < \theta < 1$ such that
    $$|\varphi(x) - \varphi(y)| \leq C_\varphi d(x, y)^\theta,\ \forall\, x, y\in M.$$
    Recall that we denote by
    $$S_\varphi F_0(p) := \sum_{i = 0}^{\pi(p) - 1}\varphi(F_0^i(p)) = \pi(p)\lambda^c(p)$$
    the Birkhoff sum of the function $\varphi$, for $p\in\mathrm{Per}(F_0)$ with period $\pi(p)$. 
    As a result,
    $$S_\varphi F_0(p_0) = \varphi(p_0) = \lambda^c(p_0) < 0 < \lambda^c(q_0) = \varphi(q_0) = S_\varphi F_0(q_0).$$
    
\noindent \textbf{Generate new periodic points for asymptotically rational independence.} 
    Now compared with the results in \cite{Xia23}, 
    we do not know whether $S_\varphi F_0(p_0)$ and $S_\varphi F_0(q_0)$ are rationally independent. 
    Therefore, we need the following asymptotic version.
    
    \begin{claim}\label{claim: rational independence}
        There exists a sequence of periodic points $\{p_{0,m}\}\subseteq\mathrm{Per}(F_0)$ such that
        $$
        (S_\varphi F_0(p_{0,m}),\ S_\varphi F_0(q_0)) \to 0, 
        \text{ and } 
        \lambda^c(p_{0,m}) \to \lambda^c(p_0) \text{ as } m \to +\infty.
        $$
    \end{claim}
    
    \begin{proof}[Proof of Claim \ref{claim: rational independence}]
        If $S_\varphi F_0(p_0)$ and $S_\varphi F_0(q_0)$ are indeed rationally independent, 
        then we just take $p_{0,m} = p_0$ for any $m\in\mathbb{N}$. 
        Now we assume that $S_\varphi F_0(p_0)$ and $S_\varphi F_0(q_0)$ are rationally dependent. 
        By Lemma \ref{lemma: dense distribution}, 
        the set $\{S_\varphi F_0(p): p\in\mathrm{Per}(F_0)\}$ is dense in $\mathbb{R}$. 
        Therefore, there exists a sequence of periodic points $\{\bar{p}_{0,m}\} \subseteq \mathrm{Per}(F_0)$ such that
        \begin{align}
            2^{-m} < |S_\varphi F_0(\bar{p}_{0,m}) - S_\varphi F_0(p_0)| < 2^{-m+1}.\label{formula: dense distribution}
        \end{align}
        For fixed $m\in\mathbb{N}$, we construct $p_{0,m}$ from $\bar{p}_{0,m}$ by shadowing property.
        
        As an Anosov diffeomorphism on a nilmanifold $M_0$, $F_0$ is topologically transitive. 
        Hence $\bar{p}_{0,m}$ and $p_0$ are homoclinically related, 
        and there are points $x_0, y_0\in M_0$ with
        $$
        x_0 \in \mathcal{F}^u(p_0) \pitchfork\mathcal{F}^s(\bar{p}_{0,m}) 
        \text{ and } y_0 \in \mathcal{F}^u(\bar{p}_{0,m}) \pitchfork \mathcal{F}^s(p_0),
        $$
        see Figure \ref{figure: periodic pseudo orbit}. 
        For simplicity, we omit the picture of the other points in $\mathrm{Orb}(\bar{p}_{0,m})$.

        \begin{figure}[htbp]
			\centering
            \begin{tikzpicture}[mid arrow/.style = {decoration = {markings, mark = at position 0.5 with {\arrow{>}}}, postaction = {decorate}}]
                \draw [thick, mid arrow](-1, 5)--(6, 5);
                \draw [thick, mid arrow](5, 6)--(5, -1);
                \draw [thick, mid arrow](6, 0)--(-1, 0);
                \draw [thick, mid arrow](0, -1)--(0, 6);
                \node [circle, fill, inner sep = 1pt, label = south east: $p_0$] at (5,0){};
                \node [circle, fill, inner sep = 1pt, label = south west: $x_0$] at (0,0){};
                \node [circle, fill, inner sep = 1pt, label = north west: $\bar{p}_{0,m}$] at (0,5){};
                \node [circle, fill, inner sep = 1pt, label = north east: $y_0$] at (5,5){};
                \node [label = above: $\mathcal{F}^u(\bar{p}_{0,m})$] at (2.5, 5) {};
                \node [label = right: $\mathcal{F}^s(p_0)$] at (5, 2.5) {};
                \node [label = below: $\mathcal{F}^u(p_0)$] at (2.5, 0) {};
                \node [label = left: $\mathcal{F}^s(\bar{p}_{0,m})$] at (0, 2.5) {};
                \node [circle, fill = red, inner sep = 1pt, label = below: \textcolor{red}{$F_0^{-K_m}(y_0)$}] at (1.25, 5) {};
                \node [label = below: $\textcolor{red}{\cdots}$] at (3,4.9) {};
                \node [circle, fill = red, inner sep = 1pt] at (3.75, 5) {};
                \node [circle, fill = red, inner sep = 1pt] at (5, 3.75) {};
                \node [label = left: $\textcolor{red}{\vdots}$] at (4.6,2.5) {};
                \node [circle, fill = red, inner sep = 1pt] at (5, 1.25) {};
                \node [circle, fill = red, inner sep = 1pt, label = above: \textcolor{red}{$F_0^{-L_m}(x_0)$}] at (3.75, 0) {};
                \node [label = above: $\textcolor{red}{\cdots}$] at (2,0.1) {};
                \node [circle, fill = red, inner sep = 1pt] at (1.25, 0) {};
                \node [circle, fill = red, inner sep = 1pt] at (0, 1.25) {};
                \node [label = right: $\textcolor{red}{\vdots}$] at (0.4,2.5) {};
                \node [circle, fill = red, inner sep = 1pt] at (0, 3.75) {};
            \end{tikzpicture}
			\caption{Periodic Pseudo-orbit}
			\label{figure: periodic pseudo orbit}
		\end{figure}
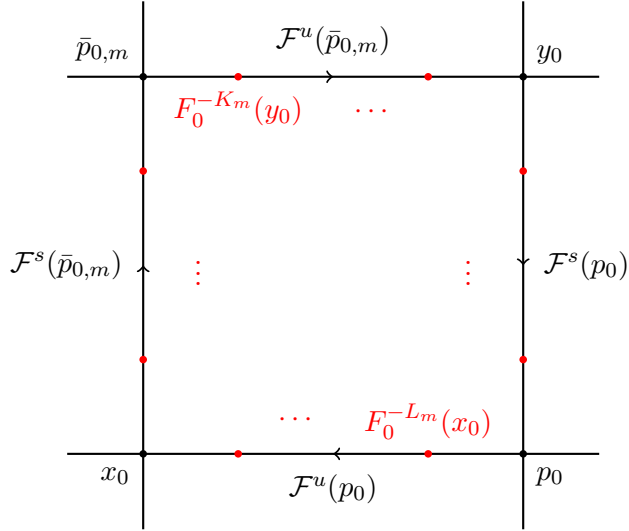
        We then consider the following pseudo-orbit segment:
        $$
        \mathcal{P}'_m := 
        \left(F_0^{-K_m}(y_0), 
        \cdots, y_0, \cdots, F_0^{L_m - 1}(y_0), F_0^{-L_m}(x_0), 
        \cdots, x_0, \cdots, F_0^{K_m - 1}(x_0)\right).
        $$
        Here $K_m,L_m\in\mathbb{N}$ are large multiples of $\pi(\bar{p}_{0,m})$ to be decided. 
        The sum of $\varphi$ over $\mathcal{P}'_m$ is
        $$
        S_\varphi(\mathcal{P}'_m):= 
        \sum_{x\in\mathcal{P}'_m}\varphi(x) =
        \sum_{i = -K_m}^{L_m - 1}\varphi(F_0^i(y_0)) + \sum_{i = -L_m}^{K_m - 1}\varphi(F_0^i(x_0)).
        $$
        Since $F_0$ is uniformly contracting (resp. expanding) in stable (resp. unstable) manifolds 
        and $\varphi$ is H\"older continuous, 
        we have the following convergent series:
        \begin{align*}
			P&=
            \sum\limits_{i = -\infty}^{-1} \left[\varphi(F_0^i(y_0)) - \varphi(F_0^i(\bar{p}_{0,m}))\right] + \sum_{i = 0}^{+\infty} \left[\varphi(F_0^i(y_0)) - \varphi(p_0)\right]\\
			&+\sum_{i = -\infty}^{-1} \left[\varphi(F_0^i(x_0)) - \varphi(p_0)\right] + \sum_{i = 0}^{+\infty} \left[\varphi(F_0^i(x_0)) - \varphi(F_0^i(\bar{p}_{0,m}))\right].
		\end{align*}
        As a result, given a sequence of positive numbers $\bar{\varepsilon}_m \to 0$, 
        there exist $K_m^0, L_m^0\in\mathbb{N}$ being multiples of $\pi(\bar{p}_{0,m})$, 
        such that when $K_m \geq K_m^0$ and $L_m \geq L_m^0$, we have
        \begin{align}
            \left|S_\varphi(\mathcal{P}'_m) - (P + 2K_m\lambda^c(\bar{p}_{0,m}) + 2L_m\lambda^c(p_0))\right| < \bar{\varepsilon}_m. \label{formula: sum over pseudo orbit}
        \end{align}
        On the other hand, by the classical \textbf{exponential Anosov shadowing lemma}, 
        there are constants $0 < \delta_0 < 1$, $C_0 \geq 1$ and $0 < \mu_0 < 1$ such that for every $0 < \delta < \delta_0$, 
        every $\delta$-pseudo-orbit $\{x_i: i_A\leq i \leq i_B\}\ (i_A \leq 0 \leq i_B)$ is exponentially shadowed by a periodic point $p$, 
        in the sense that
        $$
        d(F_0^i(p),\, x_i) < C_0\mu_0^{|i|}\delta,\ \forall\, i_A \leq i \leq i_B.
        $$
        As a result, there exist $K_m^1, L_m^1\in\mathbb{N}$ such that 
        when $K_m \geq K_m^1$ and $L_m \geq L_m^1$, 
        $\mathcal{P}'_m$ is a $\delta_0\bar{\varepsilon}_m$-pseudo-orbit, 
        exponentially shadowed by a periodic point $p_{0,m}'\in\mathrm{Per}(F_0)$. 
        In fact, it suffices to take $K_m^1$ and $L_m^1$ to be multiples of $\pi(\bar{p}_{0,m})$ and so large that
        \begin{align*}
            &\min\left\{d\Big(F_0^{L_m^1}(y_0),\, p_0\Big),
            \ d\Big(F_0^{-L_m^1}(x_0),\, p_0\Big)\right\} < \frac{1}{2}\delta_0\bar{\varepsilon}_m;\\
            &\min\left\{d\Big(F_0^{K_m^1}(x_0),\, \bar{p}_{0,m}\Big),
            \ d\Big(F_0^{-K_m^1}(y_0),\, \bar{p}_{0,m}\Big)\right\} 
            < \frac{1}{2}\delta_0\bar{\varepsilon}_m.
        \end{align*}
        
        Now we consider $S_\varphi F_0(p_{0,m}')$, 
        which is the sum of $\varphi$ over the orbit of $p_{0,m}'$. 
        Since $\varphi$ is H\"older continuous, 
        the difference between $S_\varphi F_0(p_{0,m}')$ and $S_\varphi(\mathcal{P}'_m)$ is controlled by
        \begin{align}
            \left|S_\varphi F_0(p_{0,m}') - S_\varphi(\mathcal{P}'_m)\right| 
            < C_\varphi C_0^\theta\frac{2\delta_0^\theta}{1 - \mu_0^\theta}\cdot \bar{\varepsilon}_m^\theta. 
            \label{formula: sum over periodic orbit}
        \end{align}
        
        Combining (\ref{formula: sum over pseudo orbit}) and (\ref{formula: sum over periodic orbit}), 
        there exists a constant $C \geq 1$ such that
        \begin{align}
            \left|S_\varphi F_0(p_{0,m}') - (P + 2K_m\lambda^c(\bar{p}_{0,m}) + 2L_m\lambda^c(p_0))\right| 
            < C\bar{\varepsilon}_m^\theta. 
            \label{formula: estimate of '}
        \end{align}
        Here we choose $K_m = \max\{K_m^0,\ K_m^1,\ m\pi(\bar{p}_{0,m})\}$ 
        and $L_m = \max\{L_m^0,\, L_m^1,\, m(K_m+|P|)\}$, 
        so that the central Lyapunov exponent is also controlled:
        \begin{align*}
            \left|\lambda^c(p_{0,m}') - \lambda^c(p_0)\right|
            &= \left|\frac{S_\varphi F_0(p_{0,m}')}{\pi(p_{0,m}')} - \lambda^c(p_0)\right|\\
            &\leq \frac{|P| + 2K_m|\lambda^c(\bar{p}_{0,m}) - \lambda^c(p_0)| + C\bar{\varepsilon}_m^\theta}{2K_m + 2L_m}\\
            &\to 0\ ({\rm as}\ m \to +\infty).
        \end{align*}

        In the same manner as above, we can make a slight modification on the choice of the pseudo-orbit segment. 
        Precisely, we add the periodic orbit of $\bar{p}_{0,m}$ and consider the following pseudo-orbit segment:
        $$
        \mathcal{P}''_m := \left(\bar{p}_{0,m},\, F_0(\bar{p}_{0,m}),\, 
        \cdots,\, F_0^{\pi(\bar{p}_{0,m}) - 1}(\bar{p}_{0,m})\right)\cup \mathcal{P}_m'.
        $$
        Then we have a periodic point $p_{0,m}''\in\mathrm{Per}(F_0)$ shadowing $\mathcal{P}''_m$ 
        and satisfying a similar estimate as (\ref{formula: estimate of '}):
        \begin{align}
            \left|S_\varphi F_0(p_{0,m}'') 
            - (P + 2K_m\lambda^c(\bar{p}_{0,m}) + 2L_m\lambda^c(p_0) + S_\varphi F_0(\bar{p}_{0,m}))\right| 
            < C\bar{\varepsilon}_m^\theta. \label{formula: estimate of ''}
        \end{align}
        Moreover, we also have that
        $$\left|\lambda^c(p_{0,m}'') - \lambda^c(p_0)\right| \to 0 \text{ as } m \to +\infty.$$

        Finally, we choose $p_{0,m} \in \{p_{0,m}',\, p_{0,m}''\}$ so that 
        $(S_\varphi F_0(p_{0,m}),\ S_\varphi F_0(p_0)) \to 0$ as $m \to +\infty$.
        In fact, combining (\ref{formula: estimate of '}) and (\ref{formula: estimate of ''}), we have
        \begin{align}
            \left|S_\varphi F_0(p_{0,m}'') - S_\varphi F_0(p_{0,m}') - S_\varphi F_0(\bar{p}_{0,m})\right| 
            < 2C\bar{\varepsilon}_m^\theta.
            \label{formula: difference}
        \end{align}
        Now taking $\bar{\varepsilon}_m$ so small that $2C\bar{\varepsilon}_m^\theta < {2^{-m}}$, 
        and then we have the following estimate by combining (\ref{formula: dense distribution}) and (\ref{formula: difference}):
        \begin{align}
            0 
            < \left|S_\varphi F_0(p_{0,m}'') - S_\varphi F_0(p_{0,m}') - S_\varphi F_0(p_0)\right| < 
            3\cdot2^{-m}.
            \label{formula: small distance}
        \end{align}
        Taking further $3\cdot2^{-m}<\left |S_\varphi F_0(p_0)\right|$,
        then by Lemma \ref{lemma: small distance}, there exists $p_{0,m} \in \{p_{0,m}',\, p_{0,m}''\}$ such that
        $$
        (S_\varphi F_0(p_{0,m}),\ S_\varphi F_0(p_0)) < 
        {\sqrt{3\left |S_\varphi F_0(p_0)\right|}\cdot 2^{ - m/2}}.
        $$
        Since $S_\varphi F_0(p_0)$ and $S_\varphi F_0(q_0)$ are rationally dependent, 
        there exists a sequence of positive numbers $\varepsilon_m \to 0^+$ 
        such that $(S_\varphi F_0(p_{0,m}),\ S_\varphi F_0(q_0)) < \varepsilon_m$.
        
        This completes the proof of Claim \ref{claim: rational independence}.
    \end{proof}

    In fact, we can modify the choice of $\{p_{0,m}\}$ to obtain more information about the sequence.

    \begin{claim}\label{claim: modification of sequence}
        By replacing the periodic point $p_{0,m}$ by some points in $\mathrm{Orb}(p_{0,m})$ 
        and taking a subsequence, 
        we may assume further that the sequence $\{p_{0,m}\}$ constructed in Claim \ref{claim: rational independence} satisfies 
        the following properties:
        \begin{itemize}
            \item Better Birkhoff average:
            $$
            \frac{1}{k}\sum_{i = 0}^{k - 1}\ln\left\|DF|_{E^c(F_0^i(p_{0,m}))}\right\| 
            \leq \frac{\lambda^c(p_0)}{3},\ \forall\, k\in\mathbb{N},\ \forall\, m\in\mathbb{N};
            $$
            
            \item Uniform size of contracting center: there exists $\eta_0 > 0$ such that
            $$
            |I_{p_{0,m}}^c| \geq \eta_0,\ \forall\, m\in\mathbb{N};
            $$
            
            \item Convergence: there exists $r_0\in M_0$ such that $p_{0,m}\to r_0$ as $m \to +\infty$.
        \end{itemize}
    \end{claim}
    \begin{proof}[Proof of Claim \ref{claim: modification of sequence}]
        Since $\lambda^c(p_{0,m}) \to \lambda^c(p_0) < 0$, 
        by taking a subsequence, we assume that 
        $$\lambda^c(p_{0,m}) < \frac{\lambda^c(p_0)}{2},\ \forall\, m \in\mathbb{N}.$$
        By the classical \textbf{Pliss lemma}, for each $m\in\mathbb{N}$, 
        there exists a hyperbolic time $n_m$ such that
        $$
        \frac{1}{k}\sum_{i = 0}^{k - 1}\ln\left\|DF|_{E^c(F_0^{n_m + i}(p_{0,m}))}\right\| 
        \leq \frac{\lambda^c(p_0)}{3},\ \forall\, k\in\mathbb{N}.
        $$
        Hence we can replace $p_{0,m}$ by $F_0^{n_m}(p_{0,m})\in\mathrm{Orb}(p_{0,m})$. 
        Since $p_{0,m}$ is a periodic point, 
        this replacement does not influence the properties that we need in Claim \ref{claim: rational independence}.

        Now we show that the length of the contracting center $I_{p_{0,m}}^c$ has a positive lower bound with respect to $m\in\mathbb{N}$. 
        Since $p_{0,m}$ now admits a better Birkhoff average, 
        there exists a maximal closed interval $I_m$ containing $p_{0,m}$ as an endpoint, 
        lying in $I_{p_{0,m}}^c$, satisfying
        $$
        \bar\varphi_k(x) \leq \frac{\lambda^c(p_0)}{6},\ \forall\, 1\leq k \leq \pi(p_{0,m}),\ \forall\, x\in I_m,
        $$
        where
        $$
        \bar\varphi_k(x) := \frac{1}{k}\ln\left\|DF^k|_{E^c(x)}\right\| = \frac{1}{k}\sum_{i = 0}^{k - 1}\varphi(F^i(x)).
        $$
        Then for $x,y\in I_m$, we have
        \begin{align*}
            |\bar\varphi_k(x) - \bar\varphi_k(y)|
            &\leq \frac{1}{k}\sum_{i = 0}^{k - 1}\left|\varphi(F^i(x)) - \varphi(F^i(y))\right|
            \leq \frac{1}{k}\sum_{i = 0}^{k - 1}C_\varphi  d(F^i(x),   F^i(y))^\theta\\
            &\leq \frac{1}{k}\sum_{i = 0}^{k - 1}C_\varphi d_c(F^i(x), F^i(y))^\theta
            \leq \frac{1}{k}\sum_{i = 0}^{k - 1}C_\varphi \max_{z\in I_m}\|DF^i|_{E^c(z)}\|^\theta d_c(x, y)^\theta\\
            &\leq \frac{1}{k}\sum_{i = 0}^{k - 1}C_\varphi e^{\frac{1}{6}\lambda^c(p_0)i\theta} d_c(x, y)^\theta
            \leq \frac{1}{k}\cdot\frac{C_\varphi}{1 - e^{\frac{1}{6}\lambda^c(p_0)\theta}} d_c(x, y)^\theta.
        \end{align*}
        Let $\widetilde{p}_{0,m}$ be the other endpoint of $I_m$. 
        Since $I_m$ is maximal, 
        there exists $1\leq k_m \leq \pi(p_{0,m})$ such that
        \begin{align*}
            \frac{\lambda^c(p_0)}{6} = \bar\varphi_{k_m}(\widetilde{p}_{0,m})
            &\leq \bar\varphi_{k_m}(p_{0,m}) + |\bar\varphi_{k_m}(\widetilde{p}_{0,m}) - \bar\varphi_{k_m}(p_{0,m})|\\
            &\leq \frac{\lambda^c(p_0)}{3} + \frac{1}{k_m}\frac{C_\varphi}{1 - e^{\frac{1}{6}\lambda^c(p_0)\theta}}|I_m|^\theta\\
            &\leq \frac{\lambda^c(p_0)}{3} + \frac{C_\varphi}{1 - e^{\frac{1}{6}\lambda^c(p_0)\theta}}|I_m|^\theta.
        \end{align*}
        As a result,
        $$
        |I_m| \geq \left[-\frac{1}{6}\lambda^c(p_0)\cdot\frac{1}{C_\varphi}\left(1 - e^{\frac{1}{6}\lambda^c(p_0)\theta}\right)\right]^{\frac{1}{\theta}} > 0.
        $$
        It follows that $p_{0,m}$ has a uniform size of contracting center.
        
        Finally, since $M_0$ is compact, 
        we can take a subsequence to ensure that $p_{0,m}\to r_0$ for some $r_0\in M_0$.
    \end{proof}
    
\noindent \textbf{Show the density of the stable and unstable sets.} 
    Next, we show that the intersection of the stable and unstable sets implies their density.

    \begin{claim}\label{claim: density of s and u}
        The stable and unstable sets $W^s(p_0),W^u(q_0),W^u(p_1),W^s(q_1)$ are dense in $M$.
    \end{claim}

    \begin{proof}[Proof of Claim \ref{claim: density of s and u}]
        For convenience we define $\widehat{p}_0 := \pi_1(p_0)$ and $\widehat{p}_1 := \pi_0(p_1)$. 
        Recall that $W^s(p_0)\pitchfork W^u(p_1)$ contains a central interval. 
        Assume that the interval lies in a central leaf with endpoints $z_0\in M_0$ and $z_1\in M_1$, 
        then there exists an open central interval $J \subseteq \mathcal{F}^c(z_0)$ with an endpoint $z_0$, 
        lying in $W^s(p_0)$ and containing an interval in $W^u(p_1)$, see Figure \ref{figure: density of stable set}.
        \begin{figure}[htbp]
			\centering
            \begin{tikzpicture}[mid arrow/.style = {decoration = {markings, mark = at position 0.5 with {\arrow{>}}}, postaction = {decorate}}]
                \draw [thick, mid arrow](6, 0)--(3, 2);
                \draw [thick, mid arrow](3, 2)--(0, 0);
                \draw [thick, mid arrow](6, 5)--(3, 7);
                \draw [thick, mid arrow](3, 7)--(0, 5);
                \draw [thick, mid arrow](0, 3)--(0, 0);
                \draw [thick, mid arrow](6, 5)--(6, 2);
                \draw [thick, mid arrow](6, 2)--(3, 4);
                \draw [thick, mid arrow](3, 5)--(0, 3);
                \draw [thick](0, 5)--(0, 3);
                \draw [thick](6, 2)--(6, 0);
                \draw [thick](3, 7)--(3, 5);
                \draw [thick, color = red] (3, 5)--(3, 2);
                \draw [thick, color = red] (5.85, 0.1)--(5.85, 4.95);
                \node [label = left: $\textcolor{red}{J}$] at (3, 3.5) {};
                \node [label = left: $\textcolor{red}{F^{-N}(J)}$] at (5.85, 3.5) {};
                \node [label = left: $p_0$] at (0, 0) {};
                \node [label = left: $\gamma_{p_0}(s_{p_0})$] at (0, 3) {};
                \node [label = left: $\widehat{p}_0$] at (0, 5) {};
                \node [label = right: $p_1$] at (6, 5) {};
                \node [label = right: $\gamma_{p_1}(s_{p_1})$] at (6, 2) {};
                \node [label = right: $\widehat{p}_1$] at (6, 0) {};
                \node [label = center: $W^s(p_0)$] at (1.5, 2.5) {};
                \node [label = center: $W^u(p_1)$] at (4.5, 4.5) {};
                \node [label = below: $z_0$] at (3, 2) {};
                \node [label = above: $z_1$] at (3, 7) {};
            \end{tikzpicture}
            \caption{Density of $W^s(p_0)$}
            \label{figure: density of stable set}
        \end{figure}
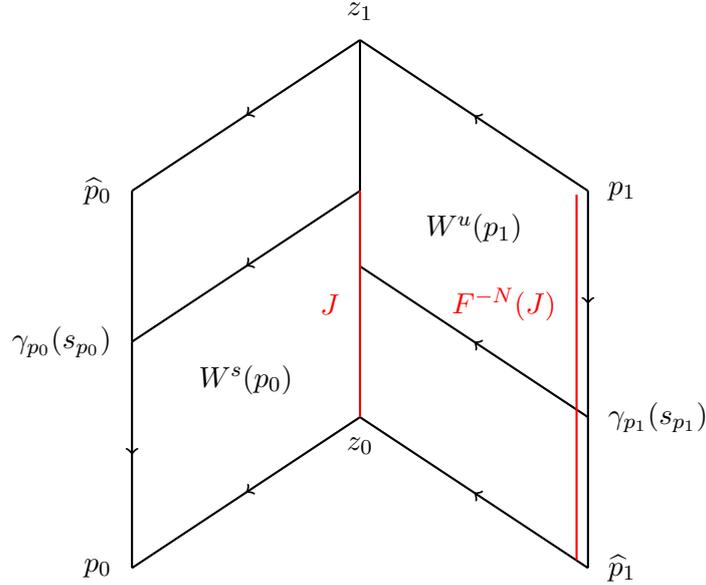

        Now consider $F^{-N}(J)$. One of its endpoints is $F_0^{-N}(z_0)$, 
        which tends to $\widehat{p}_1$; the other endpoint tends to $p_1$, since it lies in $W^u(p_1)$. 
        Moreover, $F^{-N}(J)$ is a central interval, 
        hence $F^{-N}(J)$ tends to $\mathcal{F}^c(p_1)$, in the sense that
        $$
        \forall\, x\in\mathcal{F}^c(p_1),\ \forall\, \varepsilon > 0,\ \exists\, N\in\mathbb{N},\ \exists\, x'\in F^{-N}(J),\ \text{such that}\ d(x',\, x) < \varepsilon.
        $$

        Since $F_0$ is an Anosov diffeomorphism on a nilmanifold, 
        the stable foliation restricted to $M_0$ is minimal, in the sense that each stable leaf is dense in $M_0$. 
        It follows that $\mathcal{F}^{cs}$ is minimal, 
        because $\mathcal{F}^{cs}(\cdot) = \pi_0^{-1}(\mathcal{F}^s(\pi_0(\cdot)))$.

        Now $\mathcal{F}^{cs}(p_1)$ is dense in $M$, for any non-empty open subset $U\subseteq M$, 
        there exists $x\in\mathcal{F}^c(p_1)$ such that $\mathcal{F}^s(x)\cap U\neq\varnothing$. 
        Note that
        $$
        \mathcal{F}^s(x) = \bigcup_{k\in\mathbb{N}} F^{-k}(\mathcal{F}^s_\varepsilon({F^k(x)})),
        \ \forall\, \varepsilon > 0.
        $$
        Therefore, there exists $K = K(\varepsilon)\in\mathbb{N}$ 
        such that $\mathcal{F}^s_\varepsilon({F^K(x)})\cap F^K(U)\neq\varnothing$. 
        Recall that $F^{-N}(J)$ tends to $\mathcal{F}^c(p_1)$, 
        hence there exist $N\in\mathbb{N}$ and $x'\in F^{-N}(J)$ 
        such that $\mathcal{F}^s_\varepsilon({F^K(x')})\cap F^K(U)\neq\varnothing$. 
        Now we have that $\mathcal{F}^s(F^N({F^K(x')}))\cap F^{K + N}(U)\neq\varnothing$. 
        Since ${F^N(F^K(x'))\in F^K(J)}\subseteq W^s(p_0)$, we conclude that
        $$W^s(p_0) \cap U = F^{-(K + N)}(W^s(p_0))\cap U \neq\varnothing.$$
        This shows the density of $W^s(p_0)$. 
        By the same argument, $W^u(q_0), W^u(p_1)$ and $W^s(q_1)$ are also dense in $M$.
    \end{proof}
    
\noindent \textbf{Move from $p_0$ to $p_{0,m}$.} 
    We first show that the points in $U^*$ appear near the contracting central interval $I_{p_{0,m}}^c$. 
    
    Since $\lambda^c(p_0) < 0$, we have
    $$
    W^s(p_0) = \bigcup_{k\in\mathbb{N}}F^{-k}\left(\mathcal{F}^{cs}_{loc}(p_0)\right).
    $$
    Therefore, by the density of $W^s(p_0)$, we have
    $$U^*\cap \mathcal{F}^{cs}_{loc}(p_0)\neq \varnothing,$$
    and hence we can take
    $$\widetilde{p}_0\in \pi_0(U^*)\cap \mathcal{F}^s_{loc}(p_0) \subseteq M_0.$$
    Moreover, $U^*\cap \mathcal{F}^c(\widetilde{p}_0)$ contains a closed central interval $J_{\widetilde{p}_0}^c$, 
    satisfying $\mathcal{P} := \mathcal{F}^u_{loc}(J_{\widetilde{p}_0}^c)\subseteq U^*$, 
    see Figure \ref{figure: holonomy map p}.

    \begin{figure}[htbp]
		\centering
        \begin{tikzpicture}[mid arrow/.style = {decoration = {markings, mark = at position 0.5 with {\arrow{>}}}, postaction = {decorate}}]
            \draw [thick, mid arrow, color = red] (0,1)--(1,0);
            \draw [thick, mid arrow, color = red] (1.5,-0.5)--(1,0);
            \draw [thick, mid arrow, color = green] (1,0)--(6.5,3.3);
            \draw [thick, mid arrow, color = green] (1,0)--(0.5,-0.3);
            \node [circle, fill, inner sep = 1pt, label = below: $p_0$] at (1,0) {};
            
            \draw [thick, mid arrow, color = red] (5.5,3.5)--(7,2);
            \draw [thick, mid arrow, color = red] (7.5,1.5)--(7,2);
            \draw [thick, mid arrow, color = green] (7,2)--(6.5,1.7);
            \draw [thick, mid arrow, color = green] (7,2)--(7.5,2.3);
            \node [circle, fill, inner sep = 1pt, label = below: $p_{0,m}$] at (7,2) {};

            \node [label = below: $\widetilde{p}_0$] at (0,1) {};
            \node [label = below: $z_{0,m}$] at (6,3) {};

            \draw (0,1)--(0,3);
            \draw (7,2)--(7,4);
            \draw (1,0)--(1,2);
            
            \draw [thick] (-0.5,1.2)--(-0.5,2.2);
            \draw [thick] (0,1.5)--(0,2.5);
            \draw [thick] (0.5,1.8)--(0.5,2.8);
            \draw [thick] (-0.5,1.2)--(0.5,1.8);
            \draw [thick] (-0.5,2.2)--(0.5,2.8);
            \draw (0,2) circle (27pt);
            \node [label = center: $\mathcal{P}$] at (-0.6,1) {};
            \node [label = center: $U^*$] at (0,3.2) {};
            \node [label = center: $J_{\widetilde{p}_0}^c$] at (-0.25,1.9) {};
            
            \draw [thick] (0.3,0.1)--(0.3,0.6);
            \draw [thick] (0.8,0.4)--(0.8,0.9);
            \draw [dashed, thick] (0.8,0.2)--(0.8,0.4);
            
            \draw [thick] (0.3,0.1)--(6.3,3.7);
            \draw [thick] (0.3,0.6)--(6.3,4.2);

            \draw [dashed, thick] (5.8,3.2)--(5.8,3.4);
            \draw [color = blue, thick] (5.8,3.4)--(5.8,3.9);
            \draw [thick] (6.3,3.7)--(6.3,4.2);
            \node [label = center: $\widetilde{J}_{p_{0,m}}^c$] at (5.4,4.3) {};

            \draw [dashed, thick] (5.8,3.4)--(7,2.2);
            \draw [dashed, thick] (5.8,3.9)--(7,2.7);
            \draw [color = blue, thick] (7,2.2)--(7,2.7);
            \node [label = center: $J_{p_{0,m}}^c$] at (7.5,2.4) {};

            \node [label = center: $F^{N_U}(\mathcal{P})$] at (3.3, 3) {};
        \end{tikzpicture}
        \caption{Holonomy Map from $p_0$ to $p_{0,m}$}
        \label{figure: holonomy map p}
    \end{figure}
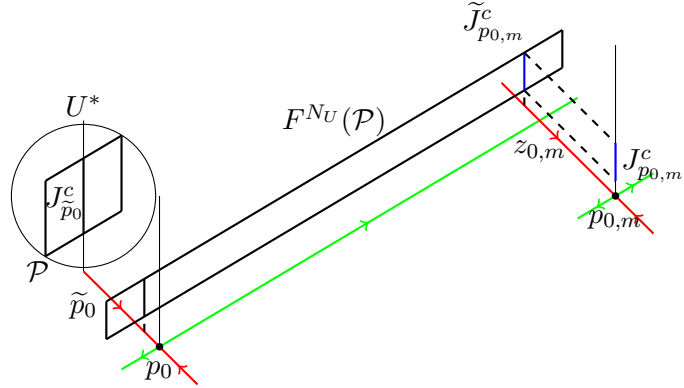

    Similarly, since $\lambda^c(q_0) > 0$ and $W^u(q_0)$ is dense in $M$, we have
    $$V^*\cap \mathcal{F}^{cu}_{loc}(q_0)\neq \varnothing,$$
    and hence we can take
    $$\widetilde{q}_0\in \pi_0(V^*)\cap \mathcal{F}^u_{loc}(q_0) \subseteq M_0.$$
    We can further require that $V^*\cap \mathcal{F}^c(\widetilde{q}_0)$ contains a closed central interval $J_{\widetilde{q}_0}^c$ 
    satisfying $\mathcal{Q} := \mathcal{F}^s_{loc}(J_{\widetilde{q}_0}^c)\subseteq V^*$, 
    see Figure \ref{figure: holonomy map q}.

    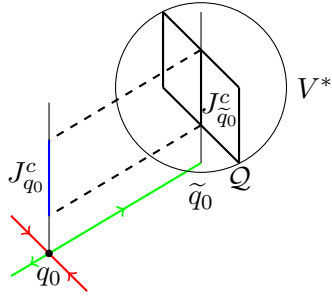
\begin{figure}[htbp]
		\centering
        \begin{tikzpicture}[mid arrow/.style = {decoration = {markings, mark = at position 0.5 with {\arrow{>}}}, postaction = {decorate}}]
            \draw [thick, mid arrow, color = red] (-0.5,0.5)--(0,0);
            \draw [thick, mid arrow, color = red] (0.5,-0.5)--(0,0);
            \draw [thick, mid arrow, color = green] (0,0)--(-0.5,-0.3);
            \draw [thick, mid arrow, color = green] (0,0)--(2,1.2);
            \node [circle, fill, inner sep = 1pt, label = below: $q_0$] at (0,0) {};
            \node [label = below: $\widetilde{q}_0$] at (2,1.2) {};
            
            \draw (0,0)--(0,2);
            \draw (2,1.2)--(2,3.2);

            \draw [thick] (2,1.7)--(2,2.7);
            \draw [thick] (1.5,2.2)--(2.5,1.2);
            \draw [thick] (1.5,3.2)--(2.5,2.2);
            \draw [thick] (1.5,2.2)--(1.5,3.2);
            \draw [thick] (2.5,1.2)--(2.5,2.2);
            \node [label = center: $J_{\widetilde{q}_0}^c$] at (2.25, 1.9) {};
            \node [label = center: $\mathcal{Q}$] at (2.5,1) {};
            
            \draw (2,2.2) circle (32pt);
            \node [label = center: $V^*$] at (3.5, 2.2) {};

            \draw [dashed, thick] (2,1.7)--(0,0.5);
            \draw [dashed, thick] (2,2.7)--(0,1.5);
            \draw [thick, color = blue] (0,0.5)--(0,1.5);
            \node [label = center: $J_{q_0}^c$] at (-0.3, 1) {};
        \end{tikzpicture}
        \caption{Holonomy Map near $q_0$}
        \label{figure: holonomy map q}
    \end{figure}
    
    Note that $\mathcal{F}^u(p_0)$ is dense in $M_0$. 
    Hence, we can choose $z_{0,m}\in \mathcal{F}^u(p_0)\pitchfork \mathcal{F}^s_{loc}(p_{0,m})$. 
    Since $p_{0,m}$ converges, we may further assume that $z_{0,m}$ also converges.
    
    In fact, by taking a subsequence we assume that $p_{0,m}$ lies in the local $su$-box of $r_0$, 
    then it suffices to ensure that $\mathcal{F}^u(p_0)\pitchfork \mathcal{F}^s_{loc}(r_0)\neq\varnothing$.
    
    It follows that there exists $N_U\in\mathbb{N}$ independent of $m\in\mathbb{N}$, such that
    $$F^{N_U}(\mathcal{P})\pitchfork \mathcal{F}^c(\mathcal{F}^s_{loc}(p_{0,m}))\neq\varnothing.$$
    In particular, $F^{N_U}(\mathcal{P})\pitchfork \mathcal{F}^c(\mathcal{F}^s_{loc}(p_{0,m}))$ contains a central interval $\widetilde{J}_{p_{0,m}}^c$, 
    whose length has a positive lower bound with respect to $m\in\mathbb{N}$.

    Let $J_{p_{0,m}}^c \subseteq\mathcal{F}^c(p_{0,m})$ be the image of $\widetilde{J}_{p_{0,m}}^c$ 
    under the holonomy map along the stable leaves, 
    restricted to $\mathcal{F}^{c}(\mathcal{F}^{s}_{loc}(p_{0,m}))$. 
    By Lemma \ref{lemma: C1-holonomy}, 
    the length of $J_{p_{0,m}}^c$ also has a positive lower bound with respect to $m\in\mathbb{N}$. 
    Moreover, since $p_{0,m}$ has a uniform size of contracting center and $J_{\widetilde{p}_0}^c$ lies in $W^s(p_0)$, 
    by taking a larger $N_U$, we can assume further that $J_{p_{0,m}}^c \subseteq I_{p_{0,m}}^c$.

    Recall that in Lemma \ref{lemma: W^s and W^u}, 
    we have a $C^{1+}$ parameterization $\gamma_{p_{0,m}}^s : [0, 1] \to \bar{I}_{p_{0,m}}^c$ for each $m\in\mathbb{N}$. 
    Since $I_{p_{0,m}}^c$ has a uniform size, we can further require that there exists $0 < \sigma_0 < 1$ such that
    $$\sigma_0 < |(\gamma_{p_{0,m}}^s)'(t)| < \sigma_0^{-1},\ \forall\, t\in[0, 1],\ \forall\, m\in\mathbb{N}.$$
    
    On the other hand, let $J_{q_0}^c\subseteq \mathcal{F}^c_{loc}(q_0)$ be the image of $J_{\widetilde{q}_0}^c$ 
    under the holonomy map along the unstable leaves, 
    restricted to $\mathcal{F}^{cu}_{loc}(q_0)$. 
    Since $\lambda^c(q_0) > 0$, we may assume further that $J_{q_0}^c\subseteq I_{q_0}^c$. 
    Similarly, we have a $C^{1+}$ parameterization $\gamma^s_{q_0}: [0, 1] \to \bar{I}_{q_0}^c$ satisfying
    $$\sigma_0 < |(\gamma^s_{q_0})'(t)| < \sigma_0^{-1},\ \forall\, t\in[0, 1].$$

\noindent\textbf{Exhibit central intersection.} 
    From now on, we fix $m\in\mathbb{N}$ so large that
    $$
    L(\varepsilon_m) 
    < \min\left\{\inf_{m\in\mathbb{N}}\big|(\gamma_{p_{0,m}}^s)^{-1}(J_{p_{0,m}}^c)\big|,
    \ \big|(\gamma_{q_0}^s)^{-1}(J_{q_0}^c)\big|\right\}.
    $$
    
    Here $L(\cdot)$ is given by Proposition \ref{proposition: center intersection} 
    and $\varepsilon_m$ is defined in Claim \ref{claim: rational independence} satisfying
    $$(S_\varphi F_0(p_{0,m}),\ S_\varphi F_0(q_0)) < \varepsilon_m.$$
    
    Without loss of generality, 
    we assume that $p_{0,m}$ is a fixed point. Fix $r_{0,m} \in \mathcal{F}^u(p_{0,m}) \pitchfork \mathcal{F}^s(q_0)$. 
    For simplicity, we denote by $\mathrm{Hol}_x^y$ 
    the holonomy map from $x$ to $y$ along stable or unstable manifolds defined on $cu$-leaves or $cs$-leaves, 
    see Figure \ref{figure: central intersection}.

    \begin{figure}[htbp]
		\centering
        \begin{tikzpicture}[mid arrow/.style = {decoration = {markings, mark = at position 0.5 with {\arrow{>}}}, postaction = {decorate}}]
            \draw [thick, mid arrow, color = green] (0,0)--(-0.5,-0.3);
            \draw [thick, mid arrow, color = green] (0,0)--(6.5,3.9);
            \draw [thick, mid arrow, color = red] (-1,1)--(0,0);
            \draw [thick, mid arrow, color = red] (0.5,-0.5)--(0,0);
            \draw (0,0)--(0,2);
            \node [circle, fill, inner sep = 1pt, label = below: $p_{0,m}$] at (0,0) {};

            \draw [thick, mid arrow, color = red] (4,4)-- (8,0);
            \draw [thick, mid arrow, color = red] (8.5,-0.5)-- (8,0);
            \draw [thick, mid arrow, color = green] (8,0)-- (7.5,-0.3);
            \draw [thick, mid arrow, color = green] (8,0)-- (9.5,0.9);
            \draw (8,0)--(8,2);
            \node [circle, fill, inner sep = 1pt, label = below: $q_0$] at (8,0) {};

            \node [circle, fill, inner sep = 1pt, label = below: $r_{0,m}$] at (5,3) {};
            \draw (5,3)--(5,5.5);

            \draw [thick, color = green] (-1,0.2)--(6,4.4);
            \draw [thick, color = red] (4.6, 4.4)--(9.2, -0.2);
            
            \draw [thick, color = blue] (0,0.3)-- (0,1.3);
            \node [label = right: $F^{k_n}(J_{p_{0,m}}^c)$] at (0, 0.3) {};
            \draw [color = blue] (0,0.3)--(5,3.3);
            \draw [color = blue] (0,1.3)--(5,4.3);

            \draw [thick, color = blue] (8,0.6)-- (8,1.8);
            \node [label = left: $F^{-l_n}(J_{q_0}^c)$] at (8, 0.6) {};
            \draw [color = blue] (8,0.6)--(5,3.6);
            \draw [color = blue] (8,1.8)--(5,4.8);

            \draw [thick, color = blue] (5,3.3)--(5,4.3);
            \draw [thick, color = blue] (5,3.6)--(5,4.8);

            \draw (-0.5,0.5)--(-0.5,2.5);
            \draw (8.625,0.375)--(8.625,2.375);
            \draw (5.125,3.875)--(5.125,5.875);
            
            \draw [dashed, color = blue] (0,0.3)--(-0.5,0.7);
            \draw [dashed, color = blue] (0,1.3)--(-0.5,1.7);
            
            \draw [thick, color = blue] (-0.5,0.7)--(-0.5,1.7);
            \draw [color = blue] (-0.5,0.7)--(5.125,4.075);
            \draw [color = blue] (-0.5,1.7)--(5.125,5.075);
            \node [label = center: $F^{k_n}(\widetilde{J}_{p_{0,m}}^c)$] at (-1.5,1.5) {};

            \draw [dashed, color = blue] (8,0.6)--(8.625,0.9);
            \draw [dashed, color = blue] (8,1.8)--(8.625,2.1);

            \draw [thick, color = blue] (8.625,0.9)--(8.625,2.1);
            \draw [color = blue] (8.625,0.9)--(5.125,4.4);
            \draw [color = blue] (8.625,2.1)--(5.125,5.6);
            \node [label = center: $F^{-l_n}(J_{\widetilde{q}_0}^c)$] at (9.5,1.5) {};

            \draw [thick, color = blue] (5.125,4.075)--(5.125,5.6);
            
            \node [circle, fill, inner sep = 1pt, label = right: $r_{0,m}^n$] at (5.125,3.875) {};
            \node [circle, fill, inner sep = 1pt, label = below: $u_n$] at (4.5,3.5) {};
            \node [circle, fill, inner sep = 1pt, label = below: $v_n$] at (5.625,3.375) {};
        \end{tikzpicture}
        \caption{Central Intersections}
        \label{figure: central intersection}
    \end{figure}
    
    \begin{claim}\label{claim: center intersection}
        There exist two strictly increasing subsequences $\{k_n\},\ \{l_n\}$ of $\mathbb{N}$ such that
        $$H\circ F^{k_n}(J_{p_{0,m}}^c)\cap F^{-l_n}(J_{q_0}^c)\neq\varnothing,$$
        where $$H := \mathrm{Hol}_{r_{0,m}}^{q_0}\circ\mathrm{Hol}_{p_{0,m}}^{r_{0,m}}.$$ 
        Moreover, there exists a constant $c > 0$ such that
        $$
        \left|H\circ F^{k_n}(J_{p_{0,m}}^c)\cap F^{-l_n}(J_{q_0}^c)\right| 
        \geq c\cdot\max\{e^{k_n\lambda^c(p_{0,m})},\ e^{-l_n\lambda^c(q_0)}\},\ \forall\, n\in\mathbb{N}.
        $$
    \end{claim}
    
    \begin{proof}[Proof of Claim \ref{claim: center intersection}]
        Note that
        $$(S_\varphi F_0(p_{0,m}),\ S_\varphi F_0(q_0)) < \varepsilon_m,$$
        which means that
        $$\left(\ln\left\|DF|_{E^c(p_{0,m})}\right\|,\ \ln\left\|DF|_{E^c(q_0)}\right\|\right) < \varepsilon_m,$$
        since $p_{0,m}$ and $q_0$ are assumed to be fixed points. 
        Equivalently, we have
        $$\left(\ln\left\|DF|_{E^c(p_{0,m})}\right\|,\ \ln\left\|DF^{-1}|_{E^c(q_0)}\right\|\right) < \varepsilon_m.$$
        The conclusion then follows from Proposition \ref{proposition: center intersection}, by taking
        \begin{align*}
            f &= (\gamma_{p_{0,m}}^s)^{-1}\circ F\circ\gamma_{p_{0,m}}^s,\\
            g &= (\gamma^s_{q_0})^{-1}\circ F^{-1}\circ \gamma^s_{q_0},\\
            h &= (\gamma^s_{q_0})^{-1}\circ H\circ\gamma_{p_{0,m}}^s.\\
        \end{align*}
        The initial intervals are $(\gamma_{p_{0,m}}^s)^{-1}(J_{p_{0,m}}^c)$ and $(\gamma^s_{q_0})^{-1}(J_{q_0}^c)$.
    \end{proof}

    As a result, there exists a constant $c > 0$ such that
    $$
    \left|{J_{p_{0,m}}^n}\cap {J_{q_0}^n}\right|
    \geq c\cdot\max\left\{e^{k_n\lambda^c(p_{0,m})},\ e^{-l_n\lambda^c(q_0)}\right\},
    \forall\, n\in\mathbb{N},
    $$
    where
    \begin{align*}
        {J_{p_{0,m}}^n} 
        &:= \mathrm{Hol}_{p_{0,m}}^{r_{0,m}}\circ F^{k_n}(J_{p_{0,m}}^c);\\
        {J_{q_0}^n}     
        &:= (\mathrm{Hol}_{r_{0,m}}^{q_0})^{-1}\circ F^{-l_n}(J_{q_0}^c).
    \end{align*}
    
    Now consider the intersection 
    between $F^{k_n}(\mathcal{F}^u_{loc}(\widetilde{J}_{p_{0,m}}^c))$ and $F^{-l_n}(\mathcal{F}^s_{loc}(J_{\widetilde{q}_0}^c))$. 
    Since $F_0$ is an Anosov diffeomorphism on a nilmanifold, for $n$ sufficiently large, 
    their projections on $M_0$, $\pi_0(F^{k_n}(\mathcal{F}^u_{loc}(\widetilde{J}_{p_{0,m}}^c)))$ 
    and $\pi_0(F^{-l_n}(\mathcal{F}^s_{loc}(J_{\widetilde{q}_0}^c)))$, 
    intersect at $r_{0,m}^n\in M_0$, which converges to $r_{0,m}$ as $n \to +\infty$. 
    Moreover, we have
    \begin{align*}
        u_n &\in \mathcal{F}^u_{loc}(r_{0,m}^n)\pitchfork \mathcal{F}^s_{loc}(r_{0,m}),\\
        v_n &\in \mathcal{F}^s_{loc}(r_{0,m}^n)\pitchfork \mathcal{F}^u_{loc}(r_{0,m}).
    \end{align*}
    For simplicity, we denote the intersections by
    \begin{align*}
        {\widehat{J}_{p_{0,m}}^n} 
        &:= F^{k_n}(\mathcal{F}^u_{loc}(\widetilde{J}_{p_{0,m}}^c))\cap \mathcal{F}^c(r_{0,m}^n);\\
        {\widehat{J}_{q_0}^n}
        &:= F^{-l_n}(\mathcal{F}^s_{loc}(J_{\widetilde{q}_0}^c))   \cap \mathcal{F}^c(r_{0,m}^n).
    \end{align*}
    Our aim is to show that 
    ${\widehat{J}_{p_{0,m}}^n}\cap {\widehat{J}_{q_0}^n}\neq\varnothing$, 
    see Figure \ref{figure: central perturbations}.

    \begin{figure}[htbp]
		\centering
        \begin{tikzpicture}[mid arrow/.style = {decoration = {markings, mark = at position 0.5 with {\arrow{>}}}, postaction = {decorate}}]
            \draw [thick, color = green] (-0.5,-0.3)--(2.5,1.5);
            \draw [thick, color = green] (-2,1.2)--(1,3);
            \draw [thick, color = red] (-2,2)--(0.5,-0.5);
            \draw [thick, color = red] (0,3.2)--(2.5,0.7);
            \node [circle, fill, inner sep = 1pt, label = below: $r_{0,m}$] at (0,0) {};
            \node [circle, fill, inner sep = 1pt, label = below: $r^n_{0,m}$] at (0.5,2.7) {};
            \node [circle, fill, inner sep = 1pt, label = below: $u_n$] at (-1.5,1.5) {};
            \node [circle, fill, inner sep = 1pt, label = below: $v_n$] at (2,1.2) {};
            \draw (0,0)--(0,2);
            \draw (-1.5,1.5)--(-1.5,3.5);
            \draw (0.5,2.7)--(0.5,4.7);
            \draw (2,1.2)--(2,3.2);

            \draw [thick, color = blue] (0,0.3)--(0,1.3);
            \draw [thick, color = blue] (0,0.6)--(0,1.6);
            \draw [color = blue] (0,0.3)--(-0.5,0);
            \draw [color = blue] (0,1.3)--(-0.5,1);
            \draw [color = blue] (0,0.6)--(0.5,0.1);
            \draw [color = blue] (0,1.6)--(0.5,1.1);
            
            \draw [thick, color = blue] (0.5,3)--(0.5,4);
            \draw [color = blue] (0.5,3)--(-2,1.5);
            \draw [color = blue] (0.5,4)--(-2,2.5);
            \draw [dashed, color = blue] (-1.5,1.8)--(0,0.1);
            \draw [dashed, color = blue] (-1.5,2.8)--(0,1.1);
            
            \draw [thick, color = blue] (0.5,3.3)--(0.5,4.5);
            \draw [color = blue] (0.5,3.3)--(2.5,1.3);
            \draw [color = blue] (0.5,4.5)--(2.5,2.5);
            \draw [dashed, color = blue] (2,1.8)--(0,0.7);
            \draw [dashed, color = blue] (2,3)--(0,1.9);

            \node [label = center: $\widehat{J}_{p_{0,m}}^n$] at (0.1,3.3) {};
            \node [label = center: $\widehat{J}_{q_0}^n$] at (0.8,3.7) {};
            \node [label = center: $J_{p_{0,m}}^n$] at (-0.4,0.7) {};
            \node [label = center: $J_{q_0}^n$] at (0.3,0.8) {};
        \end{tikzpicture}
        \caption{Dynamics near $r_{0,m}$}
        \label{figure: central perturbations}
    \end{figure}
    
    By Lemma \ref{lemma: C1-holonomy} (the $C^1$ property of holonomy maps) 
    and Lemma \ref{lemma: central twist} (the estimate of central twists), 
    there exists a constant $C\geq 1$ such that
    \begin{align*}
        &\mathrm{Hol}_{u_n}^{r_{0,m}}\circ \mathrm{Hol}_{r_{0,m}^n}^{u_n}
        ({\widehat{J}_{p_{0,m}}^n}) 
        \text{ is a } Ce^{k_n\lambda_{p,m}}\text{-perturbation of } {J_{p_{0,m}}^n};\\
        &\mathrm{Hol}_{v_n}^{r_{0,m}}\circ \mathrm{Hol}_{r_{0,m}^n}^{v_n}
        ({\widehat{J}_{q_0}^n}) 
        \text{ is a } Ce^{-l_n\lambda_q}   \text{-perturbation of } {J_{q_0}^n},
    \end{align*}
    as intervals in $\mathcal{F}^c(r_{0,m})$, where
    \begin{align*}
        \lambda_{p,m} &:= \frac{\lambda^u(p_{0,m}) - \lambda^c(p_{0,m})}{\lambda^u(p_{0,m}) - \lambda^s(p_{0,m})}\lambda^s(p_{0,m}) < \lambda^c(p_{0,m}) < 0,\\
        \lambda_q &:= \frac{-\lambda^s(q_0) + \lambda^c(q_0)}{-\lambda^s(q_0) + \lambda^u(q_0)}\lambda^u(q_0) > \lambda^c(q_0) > 0.
    \end{align*}
    Again, by the argument of central twist in Lemma \ref{lemma: central twist}, 
    for $n$ sufficiently large and $x\in \widehat{J}_q^n$, we have
    \begin{align*}
        &d_c\left(
        \mathrm{Hol}_{v_n}^{r_{0,m}}\circ \mathrm{Hol}_{r_{0,m}^n}^{v_n}(x),\, 
        \mathrm{Hol}_{u_n}^{r_{0,m}}\circ \mathrm{Hol}_{r_{0,m}^n}^{u_n}(x)\right)\\
        \leq& C\left[d_u(r_{0,m}^n,\, u_n) + d_s(u_n,\, r_{0,m}) + d_s(r_{0,m}^n,\, v_n) + d_u(v_n,\, r_{0,m})\right]\\
        \leq& C^2\left[e^{k_n\lambda^s(p_{0,m})} + e^{-l_n\lambda^u(q_0)}\right].
    \end{align*}
    Compared with the length of the intersection 
    which is controlled by
    $$
    \left|{J_{p_{0,m}}^n}\cap {J_{q_0}^n}\right|
    \geq c\cdot\max\left\{e^{k_n\lambda^c(p_{0,m})},\ e^{-l_n\lambda^c(q_0)}\right\}
    {\triangleq \rho_n},
    $$
    and the preceding perturbation errors, 
    which are \(o(\rho_n)\), 
    implied by the strict inequalities among the corresponding Lyapunov exponents,
    we conclude 
    together with the injectivity of the maps 
    $\mathrm{Hol}_{u_n}^{r_{0,m}}\circ \mathrm{Hol}_{r_{0,m}^n}^{u_n}$ and 
    $\mathrm{Hol}_{v_n}^{r_{0,m}}\circ \mathrm{Hol}_{r_{0,m}^n}^{v_n}$
    that there exists $n_0\in\mathbb{N}$, 
    such that ${\widehat{J}_{p_{0,m}}^n}\cap {\widehat{J}_{q_0}^n}\neq\varnothing$ for any $n \geq n_0$. 
    Since ${\widehat{J}_{p_{0,m}}^n}\subseteq U^*$
      and ${\widehat{J}_{q_0}^n}    \subseteq V^*$, 
    we have $U^*\cap V^*\neq\varnothing$. 
    
    This ends the proof of Theorem \ref{theorem: BI to TT}.
\end{proof}

\begin{corollary}\label{corollary: C1 and C2}
    If $F\in\mathrm{PHI}^1(M)$ is boundary interconnected, 
    then there is an open neighborhood $\mathcal{U}\subseteq \mathrm{PHI}^1(M)$ of $F$ such that 
    every $G\in\mathcal{U}\cap\mathrm{PHI}^2(M)$ is topologically transitive.
\end{corollary}

\begin{proof}
    Notice that the boundary interconnection property is $C^1$-open in $\mathrm{PHI}^1(M)$. 
    The conclusion then follows directly from Theorem \ref{theorem: BI to TT}.
\end{proof}

\section{Robust Transitivity Implies Boundary Interconnection}\label{section: RTT to BI}
In this section, we complete the proof of Theorem \ref{theorem: main theorem A} 
by showing that robust transitivity implies boundary interconnection.

\begin{theorem}\label{theorem: RTT to BI}
    If $F\in\mathrm{PHI}^k(M)\ (k \geq 2)$ is $C^k$-robustly transitive, then $F$ is boundary interconnected.
\end{theorem}

As a preparation of the proof, we need the following method to construct an attractor by perturbation.

\begin{lemma}\label{lemma: perturbation for attractor}
    Let $F\in\mathrm{PHI}^k(M)\ (k \geq 2)$. 
    If $\lambda^c(p) \leq 0$ for every $p\in\mathrm{Per}(F_0)$, 
    then for any neighborhood $\mathcal{U}\subseteq \mathrm{PHI}^k(M)$ of $F$, 
    there exists $G\in\mathcal{U}$ such that $M_0$ is an attractor of $G$.
\end{lemma}

\begin{proof}
    By Corollary \ref{corollary: central Lyapunov exponent}, 
    we have $\lambda^c(x) \leq 0$ whenever $\lambda^c(x)$ exists. 
    To show that $M_0$ is an attractor of $G$, 
    it suffices (see \cite[Lemma 2]{Cao03})
    to check that there exists a constant $\varepsilon > 0$ such that
    $$
    \{x\in M_0: \lambda^c_G(x) \leq - \varepsilon\}
    $$
    has full measure with respect to any ergodic measure of $G_0$.
    
    Let $X(s) := s(s - 1)$ be a $C^\infty$ vector field on $[0,1]$ and $\Phi_t(s)$ be the $C^\infty$ flow generated by $X$. 
    Then by abuse of notation we define
    $$\Phi_t(z, s) := (z, \Phi_t(s)): N\times[0, 1] \to N\times [0,1].$$
    It follows that $\Phi_t$ converges to identity as $t\to 0^+$, in any $C^k$ topology.

    Now consider $G := \Phi_t\circ F$. 
    Since $G$ acts in the same way as $F$ on $M_0$ and $M_1$, 
    we have $G\in\mathrm{Diff}_\partial^k(M)$ and hence when $t > 0$ is sufficiently small, $G\in\mathrm{PHI}^k(M)$.

    Note that the partially hyperbolic splitting of $G$ restricted to $M_0$ has the same stable and unstable subbundles as $F$. 
    We denote it by
    $$T_xM = E^s(x)\oplus E^c_G(x)\oplus E^u(x),\ \forall\, x\in M_0.$$

    Define $\alpha_t := \Phi'_t(0){=\partial_s\Phi_t(0)} \in (0, 1)$, and 
    $\beta(z) := f'_z(0)=\partial_s f_z(0) > 0$
    where we denote $F(z,s)=(F_s(z),f_z(s))$.
    Here, $\alpha_t=e^{-t}$ which can be solved by $\frac{d}{dt}\alpha_t=-\alpha_t$ and $\alpha_0=1$, 
    and hence $\alpha_t \in (0, 1)$ when $t>0$.
    Then we have
    \begin{align*}
        D\Phi_t(z, 0) &= \left(
        \begin{array}{cc}
            I&0\\
            0&\alpha_t 
        \end{array}\right);\\
        DF(z, 0) &= \left(
        \begin{array}{cc}
            DF_0(z)&*\\
            0&\beta(z)
        \end{array}
        \right);\\
        DG(z, 0) &= \left(
        \begin{array}{cc}
            DF_0(z)&* \\
            0&\alpha_t\beta(z)
        \end{array}
        \right).
    \end{align*}
    Let $v^c = (*, 1)^T$ and $v^c_G = (*, 1)^T$ be vectors of $E^c(z, 0)$ and $E^c_G(z, 0)$ respectively. 
    Then
    $$DF(z, 0)(v^c) = (*,\, \beta(z))^T \text{ and } DG(z, 0)(v^c_G) = (*,\, \alpha_t\beta(z))^T.$$
    Let $\|\cdot\|$ be the original induced norm in $E^c$ and $E^c_G$, and define
    \begin{align*}
        &\left\|(Z, S)^T\right\|_F := |S|,\ \forall\, (Z,S)^T\in E^c;\\
        &\left\|(Z, S)^T\right\|_G := |S|,\ \forall\, (Z,S)^T\in E^c_G.
    \end{align*}
    Since $E^c$ and $E^c_G$ are both transverse to $M_0$, which is compact, 
    the norms $\|\cdot\|_F$ and $\|\cdot\|_G$ are both equivalent with $\|\cdot\|$. 
    That is, there exists a constant $C \geq 1$ such that
    \begin{align*}
        C^{-1}\|\cdot\| \leq \|\cdot\|_F \leq C\|\cdot\|;\\
        C^{-1}\|\cdot\| \leq \|\cdot\|_G \leq C\|\cdot\|.
    \end{align*}
    Moreover, by the above discussion we have
    $$\|DF(z, 0)(v^c)\|_F = \beta(z)\|v^c\|_F \text{ and } \|DG(z, 0)(v^c_G)\|_G = \alpha_t\beta(z)\|v^c_G\|_G.$$
    As a result, for any $x\in M_0$,
    \begin{align*}
        &\|DG^n|_{E^c_G(x)}\|_G = \alpha_t^n\|DF^n|_{E^c(x)}\|_F, \text{ and hence}\\
        &C^{-4}\alpha_t^n\|DF^n|_{E^c(x)}\|\leq \|DG^n|_{E^c_G(x)}\|\leq C^{4}\alpha_t^n\|DF^n|_{E^c(x)}\|.
    \end{align*}
    
    By the classical \textbf{Oseledets Theorem}, 
    for each ergodic measure $\mu$ of $F_0 = G_0: M_0 \to M_0$, 
    there exists a $\mu$-full measure subset $B_\mu \subseteq M_0$, 
    such that for any $x\in B_\mu$, the Lyapunov exponents of $F$ and $G$ at $x$ both exist. 
    Consider
    $$B_0 := \bigcup_{\mu \text{ ergodic}}B_\mu,$$
    then $\mu(B_0) = 1$ for any ergodic measure $\mu$.
    
    Fix $x\in B_0$. We have
    \begin{align*}
        &\frac{1}{n}\ln\left\|DG^n|_{E^c_G(x)}\right\|\\
        =& \frac{1}{n}\left(\ln\left\|DG^n|_{E^c_G(x)}\right\| - \ln\left\|DF^n|_{E^c(x)}\right\|\right) + \frac{1}{n}\ln\left\|DF^n|_{E^c(x)}\right\|\\
        \in& \left[\frac{1}{n}\ln\left\|DF^n|_{E^c(x)}\right\| + \ln\alpha_t - \frac{4}{n}\ln C,\,  \frac{1}{n}\ln\left\|DF^n|_{E^c(x)}\right\| + \ln\alpha_t + \frac{4}{n}\ln C\right].
    \end{align*}
    Let $n \to +\infty$ and we have
    $$\lambda^c_G(x) = \lambda^c(x) + \ln\alpha_t \leq \ln\alpha_t < 0,\ \forall\, x\in B_0.$$
    This completes the proof of Lemma \ref{lemma: perturbation for attractor}.
\end{proof}

\begin{proof}[Proof of Theorem \ref{theorem: RTT to BI}]
    Since $F$ is $C^k$-robustly transitive, 
    there exists an open neighborhood $\mathcal{U}\subseteq\mathrm{PHI}^k(M)$ of $F$ such that 
    every $G\in\mathcal{U}$ is topologically transitive.
    
    We first prove that there exist $p_0,\ q_0\in\mathrm{Per}(F_0)$ such that $\lambda^c(p_0) < 0 < \lambda^c(q_0)$. 
    Similar conclusion holds for $F_1$. 
    For contradiction, without loss of generality, we assume that
    $$\lambda^c(p)\leq 0,\ \forall\, p\in \mathrm{Per}(F_0).$$
    
    Then by Lemma \ref{lemma: perturbation for attractor},
    there exists a $C^k$-perturbation $G\in\mathcal{U}$ such that $M_0$ is an attractor of $G$, 
    which contradicts the topological transitivity of $G$.
    
    Now take $p_0,q_0\in M_0$ and $p_1,q_1\in M_1$ such that
    $$\lambda^c(p_0) < 0 < \lambda^c(q_0) \text{ and } \lambda^c(q_1) < 0 < \lambda^c(p_1).$$
    It remains to show that $W^s(p_0)\cap W^u(p_1)\neq\varnothing$, 
    and similarly, ${W^u(q_0)}\cap W^s(q_1)\neq\varnothing$.

    First notice that
    $$
    W^s(p_0) = \bigcup_{n\in\mathbb{N}}F^{-n}(\mathcal{F}^{cs}_{loc}(p_0)) 
    \text{ and } 
    W^u(p_1) = \bigcup_{n\in\mathbb{N}}  F^n (\mathcal{F}^{cu}_{loc}(p_1)).
    $$
    Since $F$ is topologically transitive, there exists $x\in M$ such that
    $$\overline{\{F^i(x): i \geq i_0\}} = M,\ \forall\, i_0\in\mathbb{Z}.$$
    In particular, there exists $i_0\in\mathbb{Z}$ such that 
    $F^{i_0}(x)\in\mathcal{F}^s_{\varepsilon_0}(\mathcal{F}^{cu}_{loc}(p_1))$ for some $\varepsilon_0 > 0$. 
    Assume that $F^{i_0}(x)\in\mathcal{F}^s_{\varepsilon_0}(z_1)$, where $z_1\in\mathcal{F}^{cu}_{loc}(p_1)$.

    By the density of $\{F^i(x): i\geq i_0\}$, for any $\varepsilon > 0$, 
    there exists $N\in\mathbb{N}$ such that $F^{i_0 + N}(x)\in\mathcal{F}^u_\varepsilon(\mathcal{F}^{cs}_\varepsilon(p_0))$. 
    Moreover, for fixed $\varepsilon > 0$, $N$ can be arbitrarily large. 
    By taking $\varepsilon > 0$ small, we may assume that
    $$\mathcal{F}^s_\varepsilon(F^{i_0 + N}(x))\subseteq \mathcal{F}^u_{\varepsilon_0}(\mathcal{F}^{cs}_{loc}(p_0)).$$
    Now take $N$ large so that $F^N(z_1)\in\mathcal{F}^s_\varepsilon(F^{i_0 + N}(x))$. 
    It follows that $F^N(z_1)\in \mathcal{F}^u_{\varepsilon_0}(z_0)$ for some $z_0\in\mathcal{F}^{cs}_{loc}(p_0)$. 
    Since $W^u(p_1)$ is $F$-invariant and $u$-saturated, we have $z_0\in W^s(p_0)\cap W^u(p_1)$. 
    
    This completes the proof of Theorem \ref{theorem: RTT to BI}.
\end{proof}

\section{Intermingled Basins}\label{section: IB}
In this section, we prove Theorem \ref{theorem: main theorem B}. 
Recall that for $F\in\mathrm{PHI}_-^k(M)\ (k \geq 2)$, 
there are boundary SRB measures $\mu_0$ on $M_0$ and $\mu_1$ on $M_1$. 
For $i = 0,1$, since $\mu_i$ is the unique SRB measure of $F_i: M_i \to M_i$, 
we have the following results:

\begin{itemize}
    \item $\mathrm{supp}(\mu_i) = M_i$;
    
    \item $B(\mu_i)\cap M_i$ has full Lebesgue measure in $M_i$;
    
    \item For any $u$-segment $\gamma^u\subseteq M_i$, $\gamma^u\cap B(\mu_i)$ has full Lebesgue measure in $\gamma^u$;
    
    \item There exists $L_i\subseteq M_i$ with positive Lebesgue measure such that 
            $\lambda^c(x) < 0,\ \forall\, x\in L_i$, 
            and consequently, $|W^s(x)\cap\mathcal{F}^c(x)| > 0,\ \forall\, x\in L_i$;
    
    \item There exists $K_i\subseteq L_i$ with positive Lebesgue measure such that
            $$\inf_{x\in K_i}|W^s(x)\cap \mathcal{F}^c(x)| > 0.$$
\end{itemize}

\begin{theorem}\label{theorem: TT to IB}
    If $F\in\mathrm{PHI}_-^2(M)$ is topologically transitive, then $F$ admits intermingled basins.
\end{theorem}

\begin{proof}[Proof of Theorem \ref{theorem: TT to IB}]
    We first prove the following claim.
    \begin{claim}\label{claim: basin in cu box}
        For any $cu$-box $\mathcal{B}$, $\mathcal{B}\cap B(\mu_0)$ has positive Lebesgue measure in $\mathcal{B}$, 
        and further, for any $cu$-box $\mathcal{B}'$ which is $C^1$-close to $\mathcal{B}$, 
        $$\mathrm{Leb}(\mathcal{B}'\cap B(\mu_0)) \geq \frac{1}{2}\mathrm{Leb}(\mathcal{B}\cap B(\mu_0)).$$
    \end{claim}
    
    \begin{proof}[Proof of Claim \ref{claim: basin in cu box}]
        Consider the open set $U = \mathcal{F}^s_\varepsilon(\mathcal{B})$. 
        Since $F$ is topologically transitive, given any non-empty open set $V$, 
        $F^N(U)\cap V\neq\varnothing$ for infinitely many $N\in\mathbb{N}$. 
        It follows that $\cup_{n\in\mathbb{N}}F^n(\mathcal{B})$ is dense in $M$. 
        In particular, there exists $N\in\mathbb{N}$ 
        and a $cu$-box $\mathcal{P} \subseteq F^N(\mathcal{B})$ sufficiently close to $M_0$, 
        such that $\pi_0(\mathcal{P})\cap K_0$ has positive Lebesgue measure in $\pi_0(\mathcal{P})$, 
        and $x\in W^s(\pi_0(x))$ for any $x\in\mathcal{P}$ with $\pi_0(x)\in K_0$. 
        On the other hand, $\pi_0(\mathcal{P})\cap B(\mu_0)$ has full Lebesgue measure in $\pi_0(\mathcal{P})$. 
        It follows that $\mathcal{P}\cap B(\mu_0)$ has positive Lebesgue measure in $\mathcal{P}$, 
        and hence $\mathcal{B}\cap B(\mu_0)$ also has positive Lebesgue measure in $\mathcal{B}$. 
        Moreover, $B(\mu_0)$ is $s$-saturated, 
        hence for any $cu$-box $\mathcal{B}'$ which is $C^1$ close to $\mathcal{B}$, 
        $\mathrm{Leb}(\mathcal{B}'\cap B(\mu_0))$ has a uniform lower bound away from zero, given as above.
    \end{proof}
    
    Similar result holds for $B(\mu_1)$. 
    Therefore, we conclude that for any non-empty open subset $U\subseteq M$, 
    $U\cap B(\mu_0)$ and $U\cap B(\mu_1)$ both have positive Lebesgue measure.

    To show that $B(\mu_0)\cup B(\mu_1)$ has full Lebesgue measure, 
    assume for contradiction that the complement $\mathcal{C}$ has positive Lebesgue measure. 
    Take a density point $x^*\in M$ of $\mathcal{C}$. 
    Since the basins are $s$-saturated, we have
    $$
    \frac
    {\mathrm{Leb}\left(\mathcal{C} \cap \mathcal{F}_\varepsilon^{cu}(x^*)\right)}
    {\mathrm{Leb}\left(\mathcal{F}_\varepsilon^{cu}(x^*)\right)} 
    \to 1 \text{ as } \varepsilon \to 0.
    $$

    By the above argument, there exists $N_k \to +\infty$ and a $cu$-box
    $\mathcal{P}_k \subseteq F^{N_k}(\mathcal{F}_{\varepsilon_0}^{cu}(x^*))$ near $M_0$, 
    with the same size, such that $\mathcal{P}_k\cap B(\mu_0)$ consists of central leaves. 
    Moreover, we have that $F^{-N_k}(\mathcal{P}_k)\subseteq \mathcal{F}^c_{\varepsilon_0}(\mathcal{F}^u_{\varepsilon_k}(x^*))$, 
    where ${\varepsilon_k \to 0^+}$. 
    Define
    $$\mathcal{P}^*_k = F^{N_k}\left(F^{-N_k}(\mathcal{P}_k)\cap\mathcal{F}^{cu}_{\varepsilon_k}(x^*)\right).$$
    Since for each $k\in\mathbb{N}$, $\det (DF^{N_k}|_{E^{cu}})$ is nearly constant on $\mathcal{F}^{cu}_{\varepsilon_0}(x^*)$ 
    when $\varepsilon_0$ is small, we conclude that
    $$\frac{\mathrm{Leb}(\mathcal{C}\cap \mathcal{P}^*_k)}{\mathrm{Leb}(\mathcal{P}^*_k)} \to 1 \text{ as } k \to +\infty,$$
    and consequently,
    $$\frac{\mathrm{Leb}(B(\mu_0)\cap\mathcal{P}^*_k)}{\mathrm{Leb}(\mathcal{P}^*_k)} \to 0 \text{ as } k \to +\infty.$$
    Noting that $\mathcal{P}_k\cap B(\mu_0)$ consists of central leaves, 
    we also have
    $$\frac{\mathrm{Leb}(B(\mu_0)\cap\mathcal{P}_k)}{\mathrm{Leb}(\mathcal{P}_k)} \to 0 \text{ as } k \to +\infty.$$
    By taking a subsequence, we may assume that $\mathcal{P}_k$ converges to a $cu$-box $\mathcal{P}$. 
    This contradicts the fact that 
    a $cu$-box near $\mathcal{P}$ intersects $B(\mu_0)$ on a subset 
    whose Lebesgue measure has a lower bound away from zero, 
    see Claim \ref{claim: basin in cu box}.
\end{proof}

\begin{theorem}\label{theorem: RIB to BI}
    If $F\in\mathrm{PHI}_-^k(M)\ (k \geq 2)$ admits $C^k$-robustly intermingled basins, then $F$ is boundary interconnected.
\end{theorem}

\begin{proof}[Proof of Theorem \ref{theorem: RIB to BI}]
    Since $F$ admits $C^k$-robustly intermingled basins, 
    there exists an open neighborhood $\mathcal{U}\subseteq \mathrm{PHI}_-^k(M)$ of $F$, 
    such that every $G\in\mathcal{U}$ admits intermingled basins.
    
    We first show that 
    there exist $q_0\in \mathrm{Per}(F_0)$ and $p_1\in\mathrm{Per}(F_1)$ admitting positive central Lyapunov exponents.

    Assume for contradiction that $\lambda^c(p) \leq 0$, $\forall p\in\mathrm{Per}(F_0)$. 
    By Lemma \ref{lemma: perturbation for attractor}, 
    there exists $G\in\mathcal{U}$ such that $M_0$ is an attractor of $G$. 
    Moreover, $G$ acts in the same way as $F$ on $M_0$ and $M_1$, 
    hence the boundary SRB measures of $G$ are also $\mu_0$ and $\mu_1$. 
    In this case, $B_G(\mu_1)$ cannot be dense near $M_0$, contradicting the fact that $G$ admits intermingled basins.

    Now we have $q_0\in\mathrm{Per}(F_0)$ and $p_1\in\mathrm{Per}(F_1)$ with positive central Lyapunov exponents. 
    Fix $p_0\in\mathrm{Per}(F_0)$ and $q_1\in\mathrm{Per}(F_1)$ with negative central Lyapunov exponents. 
    Since $F$ admits intermingled basins, 
    $B(\mu_0)\cap \mathcal{F}^{cu}_{loc}(p_1)$ has positive Lebesgue measure in $\mathcal{F}^{cu}_{loc}(p_1)$. 
    By forward iteration, we know that there exists a $cu$-box $\mathcal{P}\subseteq W^u(p_1)$ near $p_0$, 
    since we have $\mathrm{supp}(\mu_0) = M_0$. 
    Therefore, $\mathcal{P}\pitchfork \mathcal{F}^{cs}_{loc}(p_0)\neq\varnothing$, 
    and hence $W^u(p_1)\pitchfork W^s(p_0)\neq\varnothing$. 
    Similar argument shows that $W^s(q_1)\pitchfork W^u(q_0)\neq\varnothing$. 
    Hence, $F$ is boundary interconnected.
\end{proof}

\begin{proof}[Proof of Theorem \ref{theorem: main theorem B}]
    Theorem \ref{theorem: TT to IB} shows that $(2) \implies (1)$; 
    Theorem \ref{theorem: RIB to BI} shows that $(1) \implies (3)$. 
    Finally, Corollary \ref{corollary: C1 and C2} shows that $(3) \implies (2)$.
\end{proof}

\section*{Acknowledgement}
The third author is grateful to Sichuan University for the hospitality during his visits when part of the work was done. 
The authors are grateful to anonymous referees for comments and suggestions
which are helpful to improve the presentation of this paper.
Y. Shi was partially supported by the NSFC (12526207, 12571203);
M. Xia was partially supported by the NSFC (12501238) and the Fundamental Research Funds for the Central Universities (DUT24RC(3)112).
    
\bibliographystyle{amsalpha}
\bibliography{PHI-ver2}

\end{CJK}
\end{document}